\documentclass[11pt,reqno]{amsart}

\usepackage{graphicx}

\usepackage{appendix}

\usepackage{mleftright}

\usepackage{wasysym}


\DeclareMathAlphabet{\mathpzc}{OT1}{pzc}{m}{it}
\DeclareMathAlphabet{\mathantt}{OT1}{antt}{li}{it}

\usepackage{comment}
\usepackage{enumerate}
\usepackage{esvect}
\usepackage{bm}
\usepackage{amsmath}
\usepackage{amssymb}
\usepackage{mathtools}
\newcommand{\defeq}{\vcentcolon=}
\newcommand{\eqdef}{=\vcentcolon}

\usepackage{mathrsfs}
\usepackage{color}

\usepackage[normalem]{ulem}

\usepackage{url}

\newtheorem{theorem}{Theorem}[section]
\newtheorem{prop}[theorem]{Proposition}
\newtheorem{cor}[theorem]{Corollary}
\newtheorem{lemma}[theorem]{Lemma}
\newtheorem{remark}[theorem]{Remark}
\newtheorem{definition}[theorem]{Definition}

\usepackage{bbm}

\DeclareSymbolFont{bbold}{U}{bbold}{m}{n}
\DeclareSymbolFontAlphabet{\mathbbold}{bbold}

\usepackage{stmaryrd}

\usepackage{hyperref}
\hypersetup{hidelinks,colorlinks}

\hypersetup{
    linkcolor=red,
    citecolor=blue,
    urlcolor=black,
 }

\usepackage{dsfont}

\usepackage{mleftright} 

\newcommand{\CC}{\mathbb{C}}
\newcommand{\RR}{\mathbb{R}}
\newcommand{\NN}{\mathbb{N}}

\newcommand{\ZZ}{\mathbb{Z}}

\newcommand{\ksymbol}{}

\title{Generalized polygonal number representations}

\author{Glenn Bruda}
\address{Department of Mathematics, University of Florida, Gainesville, FL 32611}
\email{glenn.bruda@ufl.edu}

\allowdisplaybreaks

\begin{document}

\begin{abstract}
    Let $r_n^{\ksymbol{k}}(N)$ be the number of representations of $N$ as the sum of $n$ generalized $k$-gonal numbers and $r_n^{\square}(N)$ be the number of representations of $N$ as the sum of $n$ squares. By modifying the Heath-Brown circle method, we prove a closed-form asymptotic relation between $r_n^{\ksymbol{k}}(N)$ and $r_n^{\square}(8(k-2)N+n(k-4)^2)$ for any $k\geq3$ and any $n\geq4$. Consequently, we determine the asymptotics of $\sum_{N\leq x}r_4^{\ksymbol{k}}(N)^2$ and, via a result of Bringmann, Jang, Kane, and Tse, prove a similar closed-form asymptotic relation between the number $r_{4,+}^{\ksymbol{k}}(N)$ of representations of $N$ as the sum of four ordinary $k$-gonal numbers and $r_4^{\square}(8(k-2)N+n(k-4)^2)$. We also show that if $4\mid k$, any strictly increasing infinite subsequence on which $r_{4,+}^{\ksymbol{k}}$ is bounded converges $2$-adically to $(k-4)^2/(4-2k)\in\ZZ_2$, supplementing a result of Meng and Sun, and if $4\nmid k$, there is no strictly increasing infinite subsequence on which $r_{4,+}^{\ksymbol{k}}$ is bounded.
\end{abstract}

\maketitle

\thispagestyle{empty}

\section{Introduction and Results}

For a nonnegative integer $m$ and $k\geq3$, the $m$\textsuperscript{th} {$k$-gonal number} $p_k(m)$ is the number of dots used to arrange a regular $k$-gon with a side length of $m$ dots. From this definition, it is straightforward to see that
\begin{align*}
    p_k(m)=\sum_{i=0}^{m-1}((k-2)i+1)=\frac{(k-2)m^2-(k-4)m}{2}.
\end{align*}
From this polynomial formula, we say that $N$ is a {generalized $k$-gonal number} if $p_k(x)=N$ for some $x\in\ZZ$. As we will frequently need to make the distinction between polygonal numbers and generalized polygonal numbers, we will often refer to $k$-gonal numbers as {ordinary} $k$-gonal numbers for emphasis and clarity.

The study of polygonal numbers as additive bases began upon Fermat conjecturing that every positive integer is the sum of $k$ ordinary $k$-gonal numbers \cite{fermat}, which is now known as the Fermat--Cauchy polygonal number theorem, following Cauchy's proof \cite{cauchy} (the cases $k=3$ and $k=4$ are known as Gauss' Eureka theorem \cite{gauss} and Lagrange's four-square theorem \cite{lagrange}, respectively, and were proven earlier). Some years after Cauchy's proof, Legendre \cite{legendre} refined the Fermat--Cauchy polygonal number theorem, showing that if $k\not\equiv0\pmod{4}$, every sufficiently large integer is the sum of four ordinary $k$-gonal numbers, and if $k\equiv 0\pmod{4}$, every sufficiently large integer is the sum of five ordinary $k$-gonal numbers and at least one of these five is $0$ or $1$. These works of Cauchy and Legendre have been wonderfully exposited and proven more efficiently in papers \cite{nathanson_cauchy} and \cite{nathanson_legendre} of Nathanson.

To this day, the study of polygonal numbers as additive bases is a rich and active area of research. In \cite{supplement_to_legendre}, Meng and Sun furthered Legendre's refinement by showing that if $k\geq8$ is divisible by four, then every integer of the form $4r^2(4^{N\phi(2r+1)}-1)/(2r+1)$ is not the sum of four ordinary $k$-gonal numbers, where $r=(k-4)/4$ and $\phi$ is Euler's phi function. In \cite{guy_AMM}, Guy provided an abundance of open problems on representing integers as sums of polygonal numbers, such as determining the minimal number $\ell_k$ of generalized $k$-gonal summands needed to obtain every positive integer, which, for $k\geq8$, is easily shown to satisfy $\ell_k\geq k-4$. Ultimately, it was shown by Sun \cite{k_8_universal} that $\ell_8=4$ and by Banerjee et al.\@ \cite{k_minus_4_universal_k_geq_10} that $\ell_k=k-4$ for all $k\geq10$. In \cite{ternary_polygonal_haensch}, Haensch and Kane handled the delicate ternary case, showing that almost every positive integer is the sum of three generalized $k$-gonal numbers if $k\not\equiv0\pmod{4}$ and $k\not\equiv 2\pmod{3}$. A very general result due to Kane and Yang \cite{finiteness_thm_kane} shows that, to determine whether every positive integer is represented by a sum of (potentially various) generalized polygonal numbers, one only needs to check a finite subset of the positive integers (the maximum of this subset is ineffective as of writing, but can be made effective upon completing the classification of ternary sums of generalized polygonal numbers that represent every positive integer). Recently, much work has been done on determining the existence of representations of integers as sums of generalized polygonal numbers with parameters with a bounded number of prime divisors; see \cite{prime_restriction_1,prime_restriction_2,prime_restriction_3}.

While the existence of representations of integers as sums of polygonal numbers has been and continues to be studied extensively, the problem of counting such representations has received relatively less attention, though interest in polygonal number representation functions has been growing recently. We define $r_n^{\ksymbol{k}}(N)$ to be the number of representations of $N$ as the sum of $n$ generalized $k$-gonal numbers, and $r_{n,+}^{\ksymbol{k}}(N)$ to be the number of representations of $N$ as the sum of $n$ ordinary $k$-gonal numbers.

In \cite{ono}, Ono, Robins, and Wahl give exact formulas for the ordinary triangular number representation function $r_{n,+}^{\ksymbol{3}}(N)$ for $n\in\{2,3,4,6,8,10,12,24\}$, recovering for the $n=4$ case the formula $r_{4,+}^{\ksymbol{3}}(N)=\sigma(2N+1)$ known previously to Legendre \cite{triangular_legendre}, where $\sigma$ is the sum-of-divisors function. In the case of four variables, Li and Wang \cite{exact_formulas_r_n_k_N} give an exact formula for $r_4^{\ksymbol{k}}(N)$ for all $k\geq3$ in terms of sums of Hurwitz class numbers. 

Since $r_n^{\ksymbol{k}}(N)$ and $r_{n,+}^{\ksymbol{k}}(N)$ count solutions to the Diophantine equation $\sum_{i=1}^{n}p_k(x_i)=N$ over $\ZZ^n$ and $\NN_0^n$, respectively, based on the available machinery, handling $r_n^{\ksymbol{k}}(N)$ is easier than handling $r_{n,+}^{\ksymbol{k}}(N)$. As remarked in \cite[pg.4]{bringmann}, we expect each $x_i$ over the set of solutions $(x_1,\dots,x_n)\in\ZZ^n$ to be asymptotically equally distributed between the negative and positive integers, giving the heuristic that $r_{n,+}^{\ksymbol{k}}(N)=2^{-n} r_{n}^{\ksymbol{k}}(N)+o(N^{n/2-1})$. This is challenging to show, but Bringmann, Jang, Kane, and Tse \cite{bringmann} recently succeeded in proving the $n=4$ case, which we will see gives some nice corollaries to our main result.

While showing a relationship between $r_n^{\ksymbol{k}}(N)$ and $r_{n,+}^{\ksymbol{k}}(N)$ is certainly difficult when $k\geq5$, the $k=3$ and $k=4$ cases are much easier. For $k=3$, since $p_3(x)=p_3(-x-1)$, we in fact have $r_{n,+}^{\ksymbol{3}}(N)=2^{-n}r_{n}^{\ksymbol{3}}(N)$ for all $n$. The $k=4$ case is not quite this simple, but we still have a nice exact relationship. Indeed, for all $n$, we have
\begin{align*}
    r_n^{\ksymbol{4}}(N)=[q^N]\left(\sum_{m\in\ZZ}q^{m^2}\right)^n=[q^N]\left(-1+2\sum_{m\geq0}q^{m^2}\right)^n=\sum_{\ell=0}^{n}{n\choose\ell}(-1)^{n-\ell}2^{\ell}r_{\ell,+}^{\ksymbol{4}}(N),
\end{align*}
where $[q^N]f(q)$ denotes the $N$\textsuperscript{th} coefficient of $f(q)$. Binomial inversion gives a similar formula for $r_{n,+}^{\ksymbol{4}}(N)$ in terms of $r_{\ell}^{\ksymbol{4}}(N)$. Since $r_{n}^{\ksymbol{4}}(N)\ll N^{n/2-1+\varepsilon}$ for all $n\geq2$, it follows that $r_{n,+}^{\ksymbol{4}}(N)=2^{-n}r_{n}^{\ksymbol{4}}(N)+O(N^{(n-3)/2+\varepsilon})$ for all $n\geq3$. These nice correspondences exist for $k=3$ and $k=4$ since $p_3$ and $p_4$ have nice symmetry between negative and positive integer arguments; however, $p_k$ does not possess such symmetry for $k\geq5$, thereby making proving $r_{n,+}^{\ksymbol{k}}(N)=2^{-n} r_{n}^{\ksymbol{k}}(N)+o(N^{n/2-1})$ difficult.

The $k=4$ case of $r_n^{\ksymbol{k}}(N)$ is, of course, just the sum of squares function, counting the number of solutions to $\sum_{i=1}^{n}x_i^2=N$ over $\ZZ^n$, and is a classical object of study. Because of the significance of this case, we give it the special notation $r_n^{\square}(N)$. The sum of $n$ squares function has been studied considerably, and exact formulas exist for several $n$. For example, Jacobi \cite{jacobi_exact_formulas} gave exact divisor-sum formulas for $r_n^{\square}(N)$ for $n\in\{2,4,6,8\}$, and exact formulas for $r_{10}^{\square}(N)$ and $r_{12}^{\square}(N)$ were given by Liouville \cite{liouville_10,liouville_12}. For $n\geq4$, we also have the asymptotic formula
\begin{align}\label{classic_asymptotic_formula_r_square}
    r_n^{\square}(N)=\frac{\pi^{n/2}}{\Gamma(n/2)}\mathfrak{S}_n^{\square}(N)N^{n/2-1}+O(N^{(n-1)/4+\varepsilon}),
\end{align}
where $\mathfrak{S}_n^{\square}(N)$ is the singular series of $r_n^{\square}(N)$. An embryonic version of this result arises as a special case of the analysis done on Waring's problem by Hardy and Littlewood \cite{waring_hardy_littlewood}, which is what prompted the development of the Hardy--Littlewood circle method. The exact version of \eqref{classic_asymptotic_formula_r_square} is these days a standard result, and can be seen, for example, by taking $F^{(0)}(\mathbf{x})=\sum_{i=1}^{n}x_i^2$ in \cite[Corollary 1]{heath_brown_circle_method}.

In \cite[pgs.305--306]{nathanson_legendre}, Nathanson remarks through quotes of Uspensky and Heaslet \cite[pg.380]{uspensky_heaslet} and Weil \cite[pg.102]{weil_book} that directly from Gauss' Eureka theorem, one can obtain the Fermat--Cauchy polygonal number theorem and Legendre's refinement thereof. In this sense, the triangular numbers are the most fundamental of the polygonal numbers. However, when moving from studying the existence of polygonal number representations to the polygonal number representation functions, it is natural to instead consider the squares as the most fundamental due to the number of clean results on the sum of squares function, motivating the desire to obtain a closed-form asymptotic relation between $r_n^{\ksymbol{k}}$ and $r_n^{\square}$. Our main result is our obtaining such a relation, provided that we have at least four variables.

\begin{theorem}\label{main_result}
    Let $n\geq4$ and $k\geq3$. For $N\in\ZZ^+$, let $X_N=8(k-2)N+n(k-4)^2$. Then the number $r_n^{\ksymbol{k}}(N)$ of representations of $N$ as the sum of $n$ generalized $k$-gonal numbers satisfies
    \begin{align*}
        r_n^{\ksymbol{k}}(N)=\frac{1+(-1)^{\lfloor n/4\rfloor}4^{n-2}T(X_N,n,k)}{2^{n-3}(k-2)^{n-1}\mathcal{N}_n(X_N,k-2)}r_n^{\square}(X_N)+O(N^{(n-1)/4+\varepsilon}),
    \end{align*}
    where $T(X_N,n,k)$ is given in closed form in Proposition \ref{2_adic_polygonal} and where
    \begin{align}\label{partial_square_singular_product}
        \mathcal{N}_n(N,L)=\left(1+(-1)^{\lfloor n/4\rfloor}f_2(N,n)\right)\prod_{\substack{p\mid L\\ p\text{~\normalfont odd}}} \left(1+p^{-n/2}f_p(N,n)\right),
    \end{align}
    and $f_p$ and $f_2$ are given in closed form in Propositions \ref{odd_adic_square} and \ref{2_adic_square}, respectively.
\end{theorem}

We note that $\mathcal{N}_n(N,L)$ is simply the product of $p$-adic densities of the sum of $n$ squares equation $\sum_{i=1}^{n}x_i^2=N$ over all $p\mid 2L$.

The factor affixed to $r_n^{\square}(X_N)$ in Theorem \ref{main_result} will often simplify greatly upon a specific choice of $n$ and $k$. For example, we have
\begin{align*}
    r_4^{\ksymbol{6}}(N)&=\frac{1}{24}r_4^{\square}(32N+16)+O(N^{3/4+\varepsilon}),\\
    r_8^{\ksymbol{5}}(N)&=\frac{1}{78705} r_8^{\square}(24N+8)+O(N^{7/4+\varepsilon}).
\end{align*}
The first relation above is particularly special because it turns out that its error is actually zero (this is not true for the second relation); that is, we have the exact identity $r_4^{\ksymbol{6}}(N)={(24)}^{-1}r_4^{\square}(32N+16)$. We defer the proof and further discussion of such exact identities to Section \ref{directions_for_future_work}.

By Theorem \ref{main_result}, for $n\geq4$, the problem of estimating the $m$\textsuperscript{th} moment of $r_n^{\ksymbol{k}}(N)$ is reduced to estimating the $m$\textsuperscript{th} moment of $r_n^{\square}(X_N)$ possibly twisted by a weight such as $(-1)^N$. We have, as an example, estimated the second moment of $r_4^{\ksymbol{k}}$ for all $k\geq3$ using this reduction. We stress that there is no barrier in particular for estimating the moments of $r_n^{\ksymbol{k}}$ for $n>4$, and that we have merely chosen our example to have $n=4$ because of the relative simplicity and because it generalizes the known asymptotics of $\sum_{N\leq x} r_4^{\square}(N)^2$.

\begin{cor}\label{moment_estimates_corollary}
    Let $k\geq3$ and $p_1,\dots,p_r$ be the odd prime divisors of $k-2$. Then
    \begin{align}\label{second_moment_of_r_4_k}
        \sum_{N\leq x} r_4^{\ksymbol{k}}(N)^2=\frac{C_k\zeta(3)}{(k-2)^4}\prod_{i=1}^{r}\frac{1-p_i^{-3}}{1-p_i^{-4}} x^3+O(x^{11/4+\varepsilon}),
    \end{align}
    where $C_k=448$ if $k\not\equiv0\pmod{4}$ and $C_k=512$ if $k\equiv 0\pmod{4}$.
\end{cor}

We suspect that the error in \eqref{second_moment_of_r_4_k} can be improved, and suggest it as a possible direction for future work. We have opted to compute a second moment (as opposed to a first moment) for our example since the estimation of the first moments is essentially a generalized Gauss circle problem, which is simpler than using Theorem \ref{main_result}. Indeed, for all $n\in\ZZ^+$, we have that $\sum_{N\leq x} r_n^{\ksymbol{k}}(N)$ equals
\begin{align*}
    \#\left\{\mathbf{y}\in\ZZ^n:\sum_{i=1}^{n}p_k(y_i)\leq x\right\}=\#\left\{\mathbf{y}\in\ZZ^n:\sum_{i=1}^{n}\left(y_i-\frac{k-4}{2(k-2)}\right)^2\leq\frac{2x}{k-2}+\frac{n(k-4)^2}{4(k-2)^2}\right\},
\end{align*}
which is equal to $(\pi^{n/2}/\Gamma(n/2+1))(2x/(k-2))^{n/2}+O(x^{(n-1)/2})$. To see that this asymptotic formula indeed holds, see, for example, \cite[Theorem 2.4 and Theorem 2.6]{lattice_asymptotic} and take the lattice $\Lambda=\ZZ^n$ and the set
\begin{align*}
    S=\left\{\mathbf{t}\in\RR^n:\sum_{i=1}^{n}\left(t_i-\frac{k-4}{2(k-2)}\right)^2\leq\frac{2x}{k-2}+\frac{n(k-4)^2}{4(k-2)^2}\right\}.
\end{align*}
Of course, if one follows the same strategy as the proof of Corollary \ref{moment_estimates_corollary}, one will obtain $\sum_{N\leq x} r_4^{\ksymbol{k}}(N)=({2\pi^2}/{(k-2)^2})x^2+O(x^{7/4+\varepsilon})$, in accordance with the above, but with a worse error estimate. However, when estimating the $m$\textsuperscript{th} moment of $r_n^{\ksymbol{k}}(N)$ where $m\geq2$, one loses the simpler generalized Gauss circle problem interpretation, hence our choice to exemplify the use of Theorem \ref{main_result} to obtain \eqref{second_moment_of_r_4_k}.

A result of Bringmann, Jang, Kane, and Tse \cite[Theorem 1.1]{bringmann} asserts that $r_{4,+}^{\ksymbol{k}}(N)=2^{-4}r_4^{\ksymbol{k}}(N)+O(N^{15/16+\varepsilon})$. By this result, we obtain the following corollary of Theorem \ref{main_result}.

\begin{cor}\label{bringmann_corollary}
    Let $k\geq3$ and $p_1,\dots, p_r$ be the odd prime divisors of $k-2$. Then the number $r_{4,+}^{\ksymbol{k}}(N)$ of representations of $N$ as the sum of four ordinary $k$-gonal numbers satisfies
    \begin{align*}
        r_{4,+}^{\ksymbol{k}}(N)=\frac{A_{N,k}}{24(k-2)^3}\left(\prod_{i=1}^{r}\frac{p_i^2}{p_i^2-1}\right)r_4^{\square}(8(k-2)N+4(k-4)^2)+O(N^{15/16+\varepsilon}),
    \end{align*}
    where $A_{N,k}=1$ if $k$ is odd, $A_{N,k}=4$ if $k\equiv 2\pmod 4$, and $A_{N,k}=8+4(-1)^{N}$ if $k\equiv0\pmod{4}$.
\end{cor}

From Corollary \ref{bringmann_corollary}, we recover as special cases \cite[Corollary 1.2]{bringmann} and \cite[Corollary 1.4]{bringmann}, which assert respectively that
\begin{align*}
    r_{4,+}^{\ksymbol{6}}(N)=\frac{1}{16}\sigma(2N+1)+O(N^{15/16+\varepsilon}),\\
    r_{4,+}^{\ksymbol{5}}(N)=\frac{1}{24}\sigma(6N+1)+O(N^{15/16+\varepsilon}).
\end{align*}
These formulas follow immediately from Corollary \ref{bringmann_corollary} upon noting the identity $r_4^{\square}(2N)=8\sum_{\substack{4\nmid d\mid 2N}}d=8\sum_{d\mid N_2}(d+2d)=24\sigma(N_2)$, where $N_2=2^{-v_2(N)}N$ is the odd part of $N$.

In addition to Corollary \ref{bringmann_corollary}, we note that, in view of Corollary \ref{moment_estimates_corollary}, the relation $r_{4,+}^{\ksymbol{k}}(N)=2^{-4}r_4^{\ksymbol{k}}(N)+O(N^{15/16+\varepsilon})$ gives the formula
    \begin{align*}
        \sum_{N\leq x} r_{4,+}^{\ksymbol{k}}(N)^2=\frac{C_k\zeta(3)}{256(k-2)^4}\prod_{i=1}^{r}\frac{1-p_i^{-3}}{1-p_i^{-4}} x^3+O(x^{47/16+\varepsilon}),
    \end{align*}
where we have the error of $O(x^{47/16+\varepsilon})$ since $\sum_{N\leq x}r_4^{k}(N)N^{15/16+\varepsilon}\ll x^{47/16+\varepsilon}$, which follows by summation by parts since $\sum_{N\leq x}r_4^{k}(N)\ll x^2$. We do remark, however, that we do not need \cite[Theorem 1.1]{bringmann} to estimate the first moment $\sum_{N\leq x} r_{4,+}^{\ksymbol{k}}(N)$. In fact, it is not too difficult to see that by modifying the generalized Gauss circle problem methodology outlined earlier, one can show that for all $n\in\ZZ^+$,
\begin{align*}
     \sum_{N\leq x} r_{n,+}^{\ksymbol{k}}(N)= \frac{\pi^{n/2}}{\Gamma(n/2+1)}\left(\frac{x}{2(k-2)}\right)^{n/2}+O(x^{(n-1)/2}).
\end{align*}

Through direct analysis of the main term from Corollary \ref{bringmann_corollary}, we obtain the following result governing which subsequences of positive integers $r_{4,+}^{\ksymbol{k}}$ may be bounded on (if any).

\begin{cor}\label{conv_2_adically_corollary}
    If $k\equiv0\pmod{4}$, then any strictly increasing infinite subsequence of positive integers on which $r_{4,+}^{\ksymbol{k}}$ is bounded converges $2$-adically to $(k-4)^2/(4-2k)\in\ZZ_2$. If $k\not\equiv0\pmod{4}$, then there is no strictly increasing infinite subsequence of positive integers on which $r_{4,+}^{\ksymbol{k}}$ is bounded.
\end{cor}

For $k\equiv0\pmod{4}$ and $r=(k-4)/4>0$, the aforementioned result of Meng and Sun \cite{supplement_to_legendre} asserts that for all $N\in\ZZ^+$,
\begin{align}\label{meng_sun_result}
    r_{4,+}^{\ksymbol{k}}\mleft(4r^2\frac{4^{N\phi(2r+1)}-1}{2r+1}\mright)=0.
\end{align}
The $k\equiv0\pmod{4}$ case of Corollary \ref{conv_2_adically_corollary} can thus be viewed as a supplement to \eqref{meng_sun_result} upon observing that the strictly increasing infinite subsequence $4r^2(4^{N\phi(2r+1)}-1)/(2r+1)$ does indeed converge $2$-adically to $(k-4)^2/(4-2k)$. It would be interesting to obtain more results on the subsequences on which $r_{4,+}^{\ksymbol{k}}$ is bounded, and we suggest it as a possible direction for future work.

In a landmark paper, Heath-Brown \cite{heath_brown_circle_method} developed a variant of the circle method for problems on quadratic forms (and in a later paper \cite{heath_brown_cubic}, cubic forms). We obtain Theorem \ref{main_result} by modifying the Heath-Brown circle method to count solutions to the sum of $n$ squares equation satisfying a fixed set of congruence equations. Setting notation, for $\mathbf{k}=(k_1,\dots,k_n)\in\ZZ^n$ and $M\in\ZZ^+$, we define $r_n^{\equiv}(N)$ to be the number of solutions $(x_1,\dots,x_n)\in\ZZ^n$ to $\sum_{i=1}^{n}x_i^2=N$ such that each $x_i\equiv k_i\bmod{M}$. We use the symbol $\mathbf{k}$ because we later take $\mathbf{k}=(k,\dots,k)$ in the pursuit of obtaining Theorem \ref{main_result}, but for now, we leave $\mathbf{k}\in\ZZ^n$ arbitrary. We are interested in $r_n^{\equiv}(N)$ because it allows us to count solutions to quadratic Diophantine equations of shape $\sum_{i=1}^{n}Q(x_i)=N$, where $Q(x)=ax^2+bx+c$ is integral-valued on $\ZZ$. Indeed, beginning with the equation $\sum_{i=1}^{n}(ax_i^2+bx_i+c)=N$ and completing the square gives $\sum_{i=1}^{n}\left(4ax_i+2b\right)^2=16aN+4nb^2-16anc$, whence we have Remark \ref{r_equiv_remark}.

\begin{remark}\label{r_equiv_remark}
    Let $Q(x)=ax^2+bx+c$ be integral-valued on $\ZZ$. Let $\mathbf{k}=(2b,\dots,2b)$ and $M=4a$. Then the number $r_n^{Q}(N)$ of solutions to $\sum_{i=1}^{n}Q(x_i)=N$ over $\mathbf{x}\in\ZZ^n$ equals $r_n^{\equiv}(16aN+4nb^2-16anc)$.
\end{remark}

In Theorem \ref{modification_theorem}, we relate $r_{n}^{\equiv}(N)$ to $r_n^{\square}(N)$ for arbitrary $\mathbf{k}\in\ZZ^n$ and $M\in\ZZ^+$. So by Remark \ref{r_equiv_remark}, our methods succeed in proving an asymptotic relation between $r_n^Q(N)$ and $r_n^{\square}(N)$, which is in closed form up to evaluating the $p$-adic densities of the Diophantine equation $\sum_{i=1}^{n}Q(x_i)=N$ for $p$ dividing $4a$. Theorem \ref{main_result} then results from the special case of taking $Q(x)=p_k(x)$ and computing the polygonal $p$-adic densities. In particular, we have the following more general result due to Theorem \ref{modification_theorem} and Remark \ref{r_equiv_remark}.

\begin{theorem}\label{more_general_theorem}
    Let $Q(x)=ax^2+bx+c$ be integral-valued on $\ZZ$. Let $n\geq4$ and for $N\in\ZZ^+$, let $q_N=16aN+4nb^2-16anc$. Then the number $r_n^Q(N)$ of solutions over $\ZZ^n$ to $\sum_{i=1}^{n}Q(x_i)=N$ satisfies
    \begin{align*}
        r_n^{Q}(N)=\frac{(4a)^{-n}}{\mathcal{N}_n(q_N,2a)}\left(\prod_{p\mid 4a}\delta_p^{Q}(N,n)\right)r_n^{\square}(q_N)+O(N^{(n-1)/4+\varepsilon}),
    \end{align*}
    where $\mathcal{N}_n$ is as in \eqref{partial_square_singular_product}, and $\delta_p^{Q}(N,n)$ is the $p$-adic density
    \begin{align*}
        \delta_p^{Q}(N,n)=\lim_{d\to\infty}p^{-d(n-1)}{\#\left\{\mathbf{x}\in\left(\ZZ/p^{d+v_p(4a)}\ZZ\right)^n:\sum_{i=1}^{n}x_i^2\equiv q_N\bmod p^d,~x_i\equiv 2b\bmod p^{v_p(4a)}\right\}}.
    \end{align*}
\end{theorem}

Once one calculates each $\delta_p^{Q}(N,n)$ for $p\mid 4a$, one can use Theorem \ref{more_general_theorem} to estimate the moments of $r_n^{Q}$ if $n\geq4$, as we did in Corollary \ref{moment_estimates_corollary} for the second moment of $r_4^{\ksymbol{k}}$ using Theorem \ref{main_result}. However, we remark that, once again, the first moment can be handled by the same generalized Gauss circle problem methodology as before to obtain $\sum_{N\leq x}r_n^Q(N)= (\pi^{n/2}/\Gamma(n/2+1))(x/a)^{n/2}+O(x^{(n-1)/2})$ for all $n\in\ZZ^+$.

In Section \ref{modifying_the_heath_brown Circle Method}, we modify the Heath-Brown circle method to obtain Theorem \ref{modification_theorem}, relating $r_n^{\equiv}(N)$ to $r_n^{\square}(N)$. In Section \ref{evaluating_the_p_adic_densities}, we compute the square $p$-adic densities (which we note may be of independent interest) and the polygonal $p$-adic densities. Then, in Section \ref{section_proof_of_results}, we prove Theorems \ref{main_result} and \ref{more_general_theorem} and Corollaries \ref{moment_estimates_corollary}, \ref{bringmann_corollary}, and \ref{conv_2_adically_corollary}. Finally, we conclude with some directions for future work in Section \ref{directions_for_future_work}.

\section{Modifying the Heath-Brown Circle Method}\label{modifying_the_heath_brown Circle Method}

We proceed by merging the ideas of works \cite{heath_brown_circle_method} and \cite{heath_brown_congruence_restriction} of Heath-Brown. Letting $F$ be a positive-definite quadratic form in $n\geq4$ variables, \cite[Corollary 1]{heath_brown_circle_method} states that the number $r_F(N)$ of solutions to $F(\mathbf{x})=N$ satisfies
\begin{align}\label{pos_def_form_formula}
    r_F(N)=\left(\sigma_{\infty}\prod_{p}\sigma_p^{F}(N)\right)N^{n/2-1}+O(N^{(n-1)/4+\varepsilon}),
\end{align}
where $\sigma_{\infty}$ is the singular integral and $\sigma_p^{F}(N)$ is the $p$-adic density of the Diophantine equation $F(\mathbf{x})=N$. This is essentially the same result (albeit with better error) as Kloosterman \cite{kloosterman_refinement} gave in his refinement of the Hardy--Littlewood circle method, prompting the introduction of the Kloosterman sum. 

By Remark \ref{r_equiv_remark}, we are interested in counting solutions to the sum of $n$ squares problem that satisfy a fixed set of congruence equations. The principal upside of Heath-Brown's circle method, as it pertains to our work, is that the methods used are conducive to modification by imposing congruence conditions upon the solution set of the quadratic Diophantine equation in question. The essential question is then: for $F(\mathbf{x})=\sum_{i=1}^{n}x_i^2$, how does the congruence modification affect formula \eqref{pos_def_form_formula}? This is where Heath-Brown's other work, \cite{heath_brown_congruence_restriction}, comes into play. In \cite{heath_brown_congruence_restriction}, the number of solutions $\mathbf{x}\in \mathcal{B}\subseteq\ZZ^n$ to $F(\mathbf{x})=0$ is explored, where $\mathcal{B}$ is a box of size depending on a variable $B$ and whose elements all satisfy a fixed $n$-dimensional congruence equation. However, \cite{heath_brown_congruence_restriction} uses a traditional form of the circle method without a Kloosterman refinement, thus requiring $n\geq5$. 

By using some of the methods from \cite{heath_brown_congruence_restriction}, we modify formula \eqref{pos_def_form_formula} from \cite{heath_brown_circle_method} to count solutions satisfying a fixed $n$-dimensional congruence equation for the form $F(\mathbf{x})=\sum_{i=1}^{n}x_i^2$, yielding an asymptotic formula for $r_n^{\equiv}(N)$. We find, under the congruence modification, that the singular integral is invariant and that the Euler product of $p$-adic densities changes only at the prime factors of the modulus of the congruence (the asymptotics, including the error up to a constant, are otherwise unaffected; see Theorem \ref{modification_theorem}). From here, to obtain Theorems \ref{main_result} and \ref{more_general_theorem}, it remains to compute the $p$-adic densities of $\sum_{i=1}^{n}x_i^2=N$ with and without the congruence restriction, which is the purpose of Section \ref{evaluating_the_p_adic_densities}.

\subsection{Preliminaries and setup}

We begin by setting standard notation and conventions. Let $e(x)\defeq e^{2\pi ix}$. Per convention, we use $(a,b)$ to denote $\gcd(a,b)$. Let $\phi$ be Euler's phi function and $\omega$ be the prime omega function counting the number of distinct prime factors of its argument. For a prime $p$ and an integer $K$, let $v_p(K)$ be the $p$-adic valuation of $K$. For $T\in\ZZ^+$, when we have $x\bmod T$ in the subscript of a summation, we mean that the summation is over integers $0\leq x\leq T-1$. Similarly, $\sum_{\mathbf{x}\bmod T}\defeq\sum_{\mathbf{x}\in\{0,1,\dots,T-1\}^n}$.

Let $n\geq4$ and $k\geq3$ be fixed integers. Since the dimension $n$ is fixed throughout the paper, we use boldface symbols to denote $n$-dimensional vectors. Let $\mathbf{0}=(0,\dots,0)\in\ZZ^n$ and let $\mathbf{k}\in\ZZ^n$ and $M\in\ZZ^+$ be arbitrary. For $\mathbf{x},\mathbf{y}\in\RR^n$, set $\mathbf{x\cdot y}$ to be the usual dot product. We will use the notation $v_i$ for a vector $\mathbf{v}=(v_1,\dots,v_n)\in\ZZ^n$ to denote its $i$\textsuperscript{th} element $v_i$. We say that $\mathbf{v}$ is divisible by $K\in\ZZ$ if $K\mid v_i$ for all $1\leq i\leq n$. For $T\in\ZZ^+$, we say that $\mathbf{x}\equiv\mathbf{y}\pmod{T}$ if $x_i\equiv y_i\pmod{T}$ for all $1\leq i\leq n$. We will use big $O$ notation and the Vinogradov notation $\ll$ interchangeably as is convenient. When $\varepsilon$ appears in asymptotic notation, we mean that the asymptotic statement holds for all $\varepsilon>0$, with the implied constant depending on $\varepsilon$.

By Remark \ref{r_equiv_remark}, the problem of asymptotically relating $r_n^{Q}$ and $r_n^{\square}$ reduces to determining the proportion of solutions $\mathbf{x}\in\ZZ^n$ to the sum of $n$ squares Diophantine equation that satisfy $x_i\equiv 2b\pmod{4a}$ for all $i$. To understand how we can handle calculating this proportion, it is prudent to understand the core ideas of the Heath-Brown circle method, which can be synopsized as a weighted-approximation version of
the circle method driven by the $\delta$-method of Duke, Friedlander, and Iwaniec \cite{delta_method}.

If we wish to approximate some counting function $\sum_{\mathbf{x}} \mathbf{1}_{S}(\mathbf{x})$ for some set $S\subseteq\ZZ^n$, it is typical to instead consider the weighted counting function $\sum_{\mathbf{x}} W(\mathbf{x})$, where $W$ is some weight function\footnote{For our purposes, weight functions will be infinitely differentiable bounded functions with compact support in $\RR^n$. However, to obtain the desired asymptotic results derived from Proposition \ref{basis_of_heath_brown}, one must further require that the weight function is in the class $\mathscr{C}(S)$ defined in \cite[\S 2]{heath_brown_circle_method}.} that approximates the characteristic function $\mathbf{1}_S$ of $S$. In this vein, for $F(X_1,\dots,X_n)\in\ZZ[X_1,\dots,X_n]$, instead of approximating $\sum_{\mathbf{x}}\mathbf{1}_{\{\mathbf{y}:F(\mathbf{y})=0\}}(\mathbf{x})=\sum_{\mathbf{x}:F(\mathbf{x})=0}1$ as done by other forms of the circle method, Heath-Brown approximates $\sum_{\mathbf{x}:F(\mathbf{x})=0}W(P^{-1}\mathbf{x})$, where $P$ is a variable that we are later interested in taking asymptotically.

In \cite{delta_method}, Duke, Friedlander, and Iwaniec provide a series formula for the Kronecker delta $\delta_{m,0}$, developed as an alternative to the circle method. We refer the interested reader to \cite{app1_of_delta_method} and \cite{app2_of_delta_method} for some further applications of the $\delta$-method. After some minor tweaks to this formula for $\delta_{m,0}$ and writing $\sum_{\substack{\mathbf{x}: F(\mathbf{x})=0}}W(P^{-1}\mathbf{x})=\sum_{\substack{\mathbf{x}}}W(P^{-1}\mathbf{x})\delta_{F(\mathbf{x}),0}$, Heath-Brown \cite{heath_brown_circle_method} obtains Proposition \ref{basis_of_heath_brown}, which is the basis for all the results of the Heath-Brown circle method.

\begin{prop}[{\cite[Theorem 2]{heath_brown_circle_method}}]\label{basis_of_heath_brown}
    Let $W$ be an infinitely differentiable bounded function with compact support in $\RR^n$. Let $F(X_1,\dots,X_n)\in\ZZ[X_1,\dots,X_n]$. Let $Q>1$ and $q\in\ZZ^+$ and
    \begin{align*}
        S_q(\mathbf{c})=\sum_{\substack{a\bmod q\\ (a,q)=1}}\sum_{\substack{\mathbf{b}\bmod {q}}}e\mleft(\frac{aF_N(\mathbf{b})+\mathbf{c\cdot b}}{q}\mright)
    \end{align*}
    and
    \begin{align*}
        J_q(\mathbf{c})=\int_{\RR^n} W\mleft(P^{-1}{\mathbf{x}}{}\mright)h(Q^{-1}q,Q^{-2}F(\mathbf{x}))~e\mleft(-\frac{\mathbf{c\cdot x}}{q}\mright)d\mathbf{x}.
    \end{align*}
    Then
    \begin{align*}
        \sum_{\substack{\mathbf{x}\in\ZZ^n\\ F(\mathbf{x})=0}}W(P^{-1}\mathbf{x})=c_Q Q^{-2}\sum_{\mathbf{c}\in\ZZ^n}\sum_{q\geq1}q^{-n} S_q(\mathbf{c})J_q(\mathbf{c}),
    \end{align*}
    where $c_Q$ satisfies $c_Q=1+O_A(Q^{-A})$ for any $A>0$ and $h\in C^{\infty}(\RR^+,\RR)$ is as defined in \cite[pg.166]{heath_brown_circle_method}.
\end{prop}

We will not need to modify the function $h$ whatsoever, so we omit its discussion and analysis. It is essential to point out that for a particular choice of weight function $W$, the weighted counting function $\sum_{\substack{\mathbf{x}: F(\mathbf{x})=0}}W(P^{-1}\mathbf{x})$ is equal to the unweighted counting function $\sum_{\mathbf{x}:F(\mathbf{x})=0}1$. This is how Heath-Brown recovers formula \eqref{pos_def_form_formula}, which was previously attained by standard versions of the circle method. Imposing the condition $\mathbf{x}\equiv\mathbf{k}\pmod{M}$ on these sums with the same weight function will lead us to the desired analog of Proposition \ref{basis_of_heath_brown}, giving a formula for $r_n^{\equiv}(N)$. Since we are not handling weighted counting functions, the reason we opt for the Heath-Brown circle method is that modifying Proposition \ref{basis_of_heath_brown} by imposing the condition $\mathbf{x}\equiv\mathbf{k}\pmod{M}$ is straightforward and essentially all of the resulting estimates hold similarly to the unmodified case. It is the purpose of this section, starting with Subsection \ref{subsection_estimating_I_q}, to verify this. Before stating the analog of Proposition \ref{basis_of_heath_brown}, we set important notation.

Let $F_N(\mathbf{x})=\sum_{i=1}^{n}x_i^2-N$ and $w_0(x)$ be defined by $w_0(x)=\exp(-(1-x^2)^{-1})$ if $|x|<1$ and $w_0(x)=0$ if $|x|\geq1$. Then define $w(\mathbf{x})=e\cdot w_0(2(F_0(\mathbf{x})-1))$. This is the weight function that makes our weighted counting function equal to the unweighted counting function, and is given in \cite[pg.203]{heath_brown_circle_method}. Set $P=\sqrt{N}$. For our purposes, we will take $Q=P$ but continue to use the symbol $Q$ as a relic of the more general setting of dealing with forms of degree $d$ (wherein one should take $Q=P^{d/2}$).

\begin{prop}\label{gen_of_thm_2_series}
    Let $Q>1$ and $q\in\ZZ^+$ and
    \begin{align*}
        S_q^{\equiv}(\mathbf{c})=\sum_{\substack{a\bmod q\\ (a,q)=1}}\sum_{\substack{\mathbf{b}\bmod {Mq}\\ \mathbf{b}\equiv \mathbf{k}\bmod M}}e\mleft(\frac{aF_N(\mathbf{b})}{q}+\frac{\mathbf{c\cdot b}}{Mq}\mright)
    \end{align*}
    and
    \begin{align*}
        I_q^{\equiv}(\mathbf{c})=\int_{\RR^n} w\mleft(P^{-1}{\mathbf{x}}{}\mright)h(Q^{-1}q,Q^{-2}F_N(\mathbf{x}))~e\mleft(-\frac{\mathbf{c\cdot x}}{Mq}\mright)d\mathbf{x}.
    \end{align*}
    Then
    \begin{align}\label{series_formula_for_r_cong}
        r_{n}^{\equiv}(N)=c_Q Q^{-2}\sum_{\mathbf{c}\in\ZZ^n}\sum_{q\geq1}(Mq)^{-n} S_q^{\equiv}(\mathbf{c})I_q^{\equiv}(\mathbf{c}),
    \end{align}
    where $c_Q$ satisfies $c_Q=1+O_A(Q^{-A})$ for any $A>0$ and $h\in C^{\infty}(\RR^+,\RR)$ is as defined in \cite[pg.166]{heath_brown_circle_method}.
\end{prop}

\begin{proof}
    The proof of this theorem is a straightforward generalization of the proof of \cite[Theorem 2]{heath_brown_circle_method} for our choice of weight function. Let $\delta_{m,0}$ be the Kronecker delta. Note that $w(P^{-1}\mathbf{x})=e\cdot w_0(2(N^{-1}\sum_{i=1}^{n}x_i^2-1))$ equals $1$ if and only if $F_N(\mathbf{x})=0$. Thus,
    \begin{align*}
        r_n^{\equiv}(N)=\sum_{\substack{\mathbf{x}\in\ZZ^n\\ F_N(\mathbf{x})=0\\ \mathbf{x}\equiv \mathbf{k} \bmod M}}w(P^{-1}\mathbf{x})=\sum_{\substack{\mathbf{x}\in\ZZ^n\\ \mathbf{x}\equiv \mathbf{k} \bmod M}}w(P^{-1}\mathbf{x})\delta_{F_N(\mathbf{x}),0}.
    \end{align*}
    Invoking \cite[Theorem 1]{heath_brown_circle_method}, we obtain
    \begin{align*}
        r_n^{\equiv}(N)&=\sum_{\substack{\mathbf{x}\in\ZZ^n\\ \mathbf{x}\equiv \mathbf{k} \bmod M}}w(P^{-1}\mathbf{x})\left(c_Q Q^{-2}\sum_{q\geq1}\sum_{\substack{a\bmod q\\ (a,q)=1}}e\mleft(\frac{aF_N(\mathbf{x})}{q}\mright)h(Q^{-1}q,Q^{-2}F_N(\mathbf{x}))\right)\\
        &=c_Q Q^{-2}\sum_{q\geq1}\sum_{\substack{a\bmod q\\ (a,q)=1}}\sum_{\substack{\mathbf{x}\in\ZZ^n\\ \mathbf{x}\equiv \mathbf{k} \bmod M}}w(P^{-1}\mathbf{x}) e\mleft(\frac{aF_N(\mathbf{x})}{q}\mright)h(Q^{-1}q,Q^{-2}F_N(\mathbf{x}))\\
        &=c_Q Q^{-2}\sum_{q\geq1}\sum_{\substack{a\bmod q\\ (a,q)=1}}\sum_{\substack{\mathbf{b}\bmod Mq\\ \mathbf{b}\equiv \mathbf{k} \bmod M}}e\mleft(\frac{aF_N(\mathbf{b})}{q}\mright)\sum_{\mathbf{y}\in\ZZ^n}f(\mathbf{y}),
    \end{align*}
    where $f(\mathbf{y})=w(\mathbf{b}+Mq\mathbf{y})h(Q^{-1}q,Q^{-2}F_N(\mathbf{b}+Mq\mathbf{y}))$. The Poisson summation formula yields $\sum_{\mathbf{y}\in\ZZ^n}f(\mathbf{y})=\sum_{\mathbf{c}\in\ZZ^n}\hat{f}(\mathbf{c})$, where
    \begin{align*}
        \hat{f}(\mathbf{c})&=\int_{\RR^n}f(\mathbf{y})e(-\mathbf{c\cdot b})d\mathbf{y}=\int_{\RR^n}w(\mathbf{b}+Mq\mathbf{y})h(Q^{-1}q,Q^{-2}F_N(\mathbf{b}+Mq\mathbf{y}))e(-\mathbf{c\cdot b})d\mathbf{y}\\
        &=e\mleft(\frac{\mathbf{c\cdot b}}{Mq}\mright)\frac{1}{(Mq)^n}\int_{\RR^n} w(P^{-1}\mathbf{x})h(Q^{-1}q,Q^{-2}F_N(\mathbf{x}))e\mleft(-\frac{\mathbf{c\cdot x}}{Mq}\mright)d\mathbf{x}.
    \end{align*}
    Substituting back into our original equation then gives the desired formula.
\end{proof}

To get an analogous form of \eqref{pos_def_form_formula} for $r_n^{\equiv}(N)$, we directly estimate \eqref{series_formula_for_r_cong}. The leading-order term comes from the terms of \eqref{series_formula_for_r_cong} with $\mathbf{c}=\mathbf{0}$ and $q\leq QP^{-\varepsilon}$. In this regime, $I_q^{\equiv}(\mathbf{0})$ grows like $P^n$ times the singular integral, and as expected, the sum of $S_q^{\equiv}(\mathbf{0})$ may be approximated by the singular series (the error of this approximation by the singular series will be absorbed into the $O(N^{(n-1)/4+\varepsilon})$ error). When $q>QP^{-\varepsilon}$, the contribution is negligible for all $\mathbf{c}$, as is the contribution when $\mathbf{c}\neq\mathbf{0}$ (for all $q$), where our modification is more consequential. Nevertheless, we demonstrate that the same error estimates from \cite{heath_brown_circle_method} still hold.

\subsection{Estimating $I^{\equiv}_q(\mathbf{c})$}\label{subsection_estimating_I_q}

Let $I_q(\mathbf{c})=\int_{\RR^n} w\mleft(P^{-1}{\mathbf{x}}{}\mright)h(Q^{-1}q,Q^{-2}F_N(\mathbf{x}))e\mleft(-{\mathbf{c\cdot x}}/q\mright)d\mathbf{x}$, where $q\in\ZZ^+$ and $\mathbf{c}\in\ZZ^n$, as defined in \cite{heath_brown_circle_method}. By the relation $I_q^{\equiv}(\mathbf{c})=I_q(\mathbf{c}/M)$, the necessary asymptotic statements for $I_q(\mathbf{c})$ given in \cite{heath_brown_circle_method} hold immediately for $I_q^{\equiv}(\mathbf{c})$. For the sake of completeness and reference, we write these statements explicitly in this subsection anyway. Crucially, we note that the vector argument of $I_q$ being integral is not necessary for the analysis, and it is thus acceptable to apply Heath-Brown's results to the argument $\mathbf{c}/M$. Since the leading-order term of \eqref{series_formula_for_r_cong} comes from the terms with $\mathbf{c}=\mathbf{0}$ and $q\leq QP^{-\varepsilon}$, the simple relation between $I_q^{\equiv}$ and $I_q$ explains why the singular integral is invariant under the congruence modification: for $q\ll Q$, the singular integral arises as the coefficient of the leading-order term of $I_q^{\equiv}(\mathbf{0})=I_q(\mathbf{0})$.

The following lemma allows us to reduce the series over all $q\in\ZZ^+$ in \eqref{series_formula_for_r_cong} to $q<P$, which proves essential for a dyadic summation argument that we borrow from \cite[pg.200]{heath_brown_circle_method}.

\begin{lemma}
    Let $q\geq P$ and $\mathbf{c}\in\ZZ^n$. Then $I_q^{\equiv}(\mathbf{c})=0$.
\end{lemma}

\begin{proof}
    Since $q\geq P$, note that $q/Q=q/P\geq 1$. Suppose that $|N^{-1}\sum_{i=1}^{n}x_i^2-1|\leq 1/2$. Then
    \begin{align*}
        |Q^{-2}F_N(\mathbf{x})|=\left|\frac{1}{N}\sum_{i=1}^{n}x_i^2-1\right|\leq 1/2\leq\frac{q/Q}{2},
    \end{align*}
    whence $h(Q^{-1}q,Q^{-2}F_N(\mathbf{x}))=0$ by \cite[Lemma 4]{heath_brown_circle_method}. On the other hand, if $|N^{-1}\sum_{i=1}^{n}x_i^2-1|>1/2$, then $w(P^{-1}\mathbf{x})=e\cdot w_0(2(N^{-1}\sum_{i=1}^{n}x_i^2-1))=0$. So $I^{\equiv}_q(\mathbf{c})=0$.
\end{proof}

Let $G(\mathbf{x})=F_0(\mathbf{x})-1$. As noted by Heath-Brown \cite[pg.204]{heath_brown_circle_method}, we remark that since $w(\mathbf{x})=1+O(\varepsilon)$ for $|G(\mathbf{x})|\leq\varepsilon$,
\begin{align*}
    \sigma_{\infty}(G,w)\defeq\lim_{\varepsilon\to0}\frac{1}{2\varepsilon}\int_{|G(\mathbf{x})|\leq\varepsilon} w(\mathbf{x})d\mathbf{x}=\lim_{\varepsilon\to0}\frac{1}{2\varepsilon}\int_{|G(\mathbf{x})|\leq\varepsilon} d\mathbf{x}\eqdef \sigma_{\infty}(G).
\end{align*}
This is the singular integral. By the equality of $\sigma_{\infty}(G,w)$ and $\sigma_{\infty}(G)$, we have the following estimate by {\cite[Lemma 13]{heath_brown_circle_method}}. 

\begin{lemma}[{\cite[Lemma 13]{heath_brown_circle_method}}]\label{I_q_0_cong_est}
    Suppose that $q\ll Q$. Then for any $A>0$,
    \begin{align*}
        I^{\equiv}_q(\mathbf{0})=I_q(\mathbf{0})=P^n\left(\sigma_{\infty}(G)+O_A\left(\left(q/Q\right)^A\right)\right).
    \end{align*}
\end{lemma}

We digress that we can compute $\sigma_{\infty}(G)$ explicitly, though we do not need to for our purposes. Indeed, the integral in the definition of $\sigma_{\infty}(G)$ is simply the difference of the volumes of two $n$-dimensional balls of radii $\sqrt{1+\varepsilon}$ and $\sqrt{1-\varepsilon}$. Since $(2\varepsilon)^{-1}((1+\varepsilon)^{n/2}-(1-\varepsilon)^{n/2})$ tends to $n/2$ as $\varepsilon\to0$, we have
\begin{align*}
    \sigma_{\infty}(G)=\frac{n}{2}\frac{\pi^{n/2}}{\Gamma(n/2+1)}=\frac{\pi^{n/2}}{\Gamma(n/2)},
\end{align*}
recovering the constant given in \eqref{classic_asymptotic_formula_r_square} and as noted by Bateman \cite{bateman}, for example. This differs from the constant $\Gamma(1+1/2)^{n}/\Gamma(n/2)$ seen in Hua's result \cite{hua} on Waring's problem (regarding squares in this case) by a factor of $2^n$, which comes from the fact that we are counting solutions over $\ZZ^n$ rather than $\NN_0^n$. A more recent illustration of Hua's result may be seen in \cite[\S 20.2]{iwaniec_kowalski}.

When $q\ll Q$, the coefficient of the leading-order term of $I_q^{\equiv}(\mathbf{0})$ ultimately corresponds to the coefficient of the leading-order term of \eqref{series_formula_for_r_cong}, so determining it to be exactly the singular integral was important. However, terms with $q$ outside this range will be absorbed into the error, so it is sufficient to leave the asymptotics of $I_q^{\equiv}(\mathbf{0})$ up to a constant in this case.

\begin{lemma}[{\cite[Lemma 16]{heath_brown_circle_method}}]
    For any $q$, we have the bound
    \begin{align*}
        I^{\equiv}_q(\mathbf{0})=I_q(\mathbf{0})\ll P^n.
    \end{align*}
\end{lemma}

The next two lemmas handle bounding $I_q^{\equiv}(\mathbf{c})$ when $\mathbf{c}\neq\mathbf{0}$ (with no restriction on $q$). The first will be used when $|\mathbf{c}|>P^{\varepsilon}$, wherein having the savings factor $|\mathbf{c}|^{-A}$ for all $A>0$ is essential. 

\begin{lemma}[{\cite[Lemma 19]{heath_brown_circle_method}}]
    Suppose that $\mathbf{c}\neq0$. Then for any $A>0$,
    \begin{align*}
        I^{\equiv}_q(\mathbf{c})=I_q(\mathbf{c}/M)\ll_A P^{n+1}q^{-1}|\mathbf{c}/M|^{-A}\ll P^{n+1}q^{-1}|\mathbf{c}|^{-A}.
    \end{align*}
\end{lemma}

For $\mathbf{c}\ll P^{\varepsilon}$, we will use the following lemma. Since the bound is independent of $\mathbf{c}$, the estimation of \eqref{series_formula_for_r_cong} is essentially reduced (using a dyadic summation argument, see \cite[pg.200]{heath_brown_circle_method}) to estimating $\sum_{q\leq X}|S_q^{\equiv}(\mathbf{c})|$, which is handled by Lemma \ref{sum_of_S_qs}.

\begin{lemma}
    Suppose that $\mathbf{c}\neq0$ and $\mathbf{c}\ll P^{\varepsilon}$. Then
    \begin{align*}
        I_q^{\equiv}(\mathbf{c})\ll P^{n/2+1+\varepsilon} q^{n/2-1}.
    \end{align*}
\end{lemma}

\begin{proof}
    By \cite[Lemma 22]{heath_brown_circle_method},
    \begin{align*}
        I^{\equiv}_q(\mathbf{c})=I_q(\mathbf{c}/M)\ll P^n\left(\frac{PQ|\mathbf{c}/M|}{q^2}\right)^{\varepsilon}\left(\frac{P|\mathbf{c}/M|}{q}\right)^{1-n/2}\ll P^n\left(\frac{PQ|\mathbf{c}|}{q^2}\right)^{\varepsilon}\left(\frac{P|\mathbf{c}|}{q}\right)^{1-n/2}.
    \end{align*}
    Recalling that $Q=P$, since $\mathbf{c}\ll P^{\varepsilon}$, we have $(PQ|\mathbf{c}|/q^2)^{\varepsilon}\ll(P^{2+\varepsilon}/q^2)^{\varepsilon}$, which is $O(P^{\varepsilon})$. Additionally, since $|\mathbf{c}|^{1-n/2}\leq 1$, we have $(P|\mathbf{c}|/q)^{1-n/2}\leq P^{1-n/2} q^{n/2-1}$. So $I^{\equiv}_q(\mathbf{c})\ll P^n\cdot P^{\varepsilon}\cdot P^{1-n/2} q^{n/2-1}=P^{n/2+1+\varepsilon}q^{n/2-1}$.
\end{proof}

\subsection{Estimating $S_q^{\equiv}(\mathbf{c})$}

We now move on to estimating $S_q^{\equiv}(\mathbf{c})$. While $S_q^{\equiv}(\mathbf{c})$ is not as easily related to $S_q(\mathbf{c})$ as $I_q^{\equiv}(\mathbf{c})$ is to $I_q(\mathbf{c})$, we can still show that the bounds regarding $S_q(\mathbf{c})$ hold for $S_q^{\equiv}(\mathbf{c})$. We adopt and modify the same strategies from \cite{heath_brown_circle_method} for all but one lemma in this subsection; namely, Lemma \ref{multiplicative_property}, which exhibits the multiplicative property of $S_q^{\equiv}$, and instead follows a strategy from \cite{heath_brown_congruence_restriction}. The most difficult result of this subsection is Lemma \ref{sum_of_S_qs}, which bounds $\sum_{q\leq X}|S_q^{\equiv}(\mathbf{c})|$. A bound on $\sum_{q\leq X}|S_q^{\equiv}(\mathbf{c})|$ immediately follows from the upcoming Lemma \ref{crude_bound_on_S_q_equiv}; however, we can improve this immediate bound by an almost-square-root factor, without which we would be unable to obtain the error of $O(N^{(n-1)/4+\varepsilon})$. Obtaining this improvement is the purpose of Lemmas \ref{multiplicative_property} and \ref{prime_bound_on_S}.

The multiplicative property of $S_q^{\equiv}$ given in Lemma \ref{multiplicative_property} necessitates that we consider $S_q^{\equiv}$ with respect to moduli other than $M$. Thus, we provide the following definition.

\begin{definition}
    Let $T,q\in\ZZ^+$ and $\mathbf{c}\in\ZZ^n$. Then we define
    \begin{align*}
        S_q^{\equiv}(T,\mathbf{c})\defeq\sum_{\substack{a\bmod q\\ (a,q)=1}}\sum_{\substack{\mathbf{b}\bmod {Tq}\\ \mathbf{b}\equiv \mathbf{k}\bmod T}}e\mleft(\frac{aF_N(\mathbf{b})}{q}+\frac{\mathbf{c\cdot b}}{Tq}\mright).
    \end{align*}
    When $T=M$, we write $S_q^{\equiv}(T,\mathbf{c})=S_q^{\equiv}(\mathbf{c})$, in agreement with the notation in Proposition \ref{gen_of_thm_2_series}.
\end{definition}

We will use the following lemma to bound \eqref{series_formula_for_r_cong} when $|c|>P^{\varepsilon}$ and to prove Lemma \ref{sum_of_S_qs}. The proof follows by an application of the Cauchy--Schwarz inequality to the dot product.

\begin{lemma}\label{crude_bound_on_S_q_equiv}
    For any $\mathbf{c}\in\ZZ^n$ and $T\in\ZZ^+$, we have $S_q^{\equiv}(T,\mathbf{c})\ll q^{1+n/2}$.
\end{lemma}

\begin{proof}
    This result and its proof are direct modifications of \cite[Lemma 25]{heath_brown_circle_method}. By the Cauchy--Schwarz inequality,
    \begin{align*}
        &|S_q^{\equiv}(T,\mathbf{c})|^2=\left|\sum_{\substack{a\bmod q\\ (a,q)=1}}\sum_{\substack{\mathbf{b}\bmod {Tq}\\ \mathbf{b}\equiv \mathbf{k}\bmod T}}e\mleft(\frac{aF_N(\mathbf{b})}{q}+\frac{\mathbf{c\cdot b}}{Tq}\mright)\right|^2\\
        &\leq \left(\sum_{\substack{a\bmod q\\ (a,q)=1}} 1\right)\left(\sum_{\substack{a\bmod q\\ (a,q)=1}}\left(\sum_{\substack{\mathbf{b}\bmod {Tq}\\ \mathbf{b}\equiv \mathbf{k}\bmod T}}e\mleft(\frac{aF_N(\mathbf{b})}{q}+\frac{\mathbf{c\cdot b}}{Tq}\mright)\right)^2\right)\\
        &=\phi(q)\sum_{\substack{a\bmod q\\ (a,q)=1}}\left(\sum_{\substack{\mathbf{u}\bmod {Tq}\\ \mathbf{u}\equiv \mathbf{k}\bmod T}}e\mleft(\frac{aF_N(\mathbf{u})}{q}+\frac{\mathbf{c\cdot u}}{Tq}\mright)\right)\overline{\left(\sum_{\substack{\mathbf{v}\bmod {Tq}\\ \mathbf{v}\equiv \mathbf{k}\bmod T}}e\mleft(\frac{aF_N(\mathbf{v})}{q}+\frac{\mathbf{c\cdot v}}{Tq}\mright)\right)}\\
        &=\phi(q)\sum_{\substack{a\bmod q\\ (a,q)=1}}\sum_{\substack{\mathbf{u},\mathbf{v}\bmod{Tq}\\ \mathbf{u},\mathbf{v}\equiv\mathbf{k}\bmod{T}}}e\mleft(\frac{a(F_N(\mathbf{u})-F_N(\mathbf{v}))}{q}+\frac{\mathbf{c\cdot\phantom{}}(\mathbf{u}-\mathbf{v})}{Tq}\mright).
    \end{align*}
    By the substitution $\mathbf{u}=\mathbf{d}+\mathbf{v}$, the above is equal to
    \begin{align*}
        &\phantom{=}\phi(q)\sum_{\substack{a\bmod q\\ (a,q)=1}}\sum_{\substack{\mathbf{d},\mathbf{v}\bmod{Tq}\\ \mathbf{v}\equiv\mathbf{k}\bmod{T}\\ T\mid \mathbf{d} }}e\mleft(\frac{a(F_N(\mathbf{\mathbf{d}+\mathbf{v}})-F_N(\mathbf{v}))}{q}+\frac{\mathbf{c\cdot d}}{Tq}\mright)\\
        &=\phi(q)\sum_{\substack{a\bmod q\\ (a,q)=1}}\sum_{\substack{\mathbf{d}\bmod{Tq}\\ T\mid \mathbf{d}}} e\mleft(\frac{aF_0(\mathbf{d})}{q}+\frac{\mathbf{c\cdot d}}{Tq}\mright)\sum_{\substack{\mathbf{v}\bmod{Tq}\\ \mathbf{v}\equiv \mathbf{k}\bmod{T}}}e\mleft(\frac{2a\mathbf{v\cdot d}}{q}\mright)\\
        &=\phi(q)\sum_{\substack{a\bmod q\\ (a,q)=1}}\sum_{\substack{\mathbf{d}\bmod{Tq}\\ T\mid \mathbf{d}}} e\mleft(\frac{aF_0(\mathbf{d})}{q}+\frac{\mathbf{c\cdot d}}{Tq}\mright)\sum_{\substack{\mathbf{v}\bmod{q}}}e\mleft(\frac{2a(T\mathbf{v}+\mathbf{k})\mathbf{\phantom{}\cdot d}}{q}\mright)\\
        &=\phi(q)\sum_{\substack{a\bmod q\\ (a,q)=1}}\sum_{\substack{\mathbf{d'}\bmod{q}}} e\mleft(\frac{a(F_0(T\mathbf{d'})+2T\mathbf{k\cdot d'})+\mathbf{c\cdot d'}}{q}\mright)\sum_{\substack{\mathbf{v}\bmod{q}}}e\mleft(\frac{2aT^2\mathbf{v}\mathbf{\phantom{}\cdot d'}}{q}\mright)
    \end{align*}
    Since $(a,q)=1$, the inner summation over $\mathbf{v}$ equals $q^n$ when $q\mid 2T^2\mathbf{d'}$, and is zero otherwise. Since the number of $\mathbf{d'} \pmod q$ satisfying this property is bounded in $q$, there exists some $K$ that is constant in $q$ such that the above equation is at most
    \begin{align*}
        \phi(q)\sum_{\substack{a\bmod q\\ (a,q)=1}} K q^n=K\phi(q)^2 q^n\ll q^{2+n}.
    \end{align*}
    So $S_q^{\equiv}(T,\mathbf{c})\ll q^{1+n/2}$.
\end{proof}

The following lemma demonstrates the multiplicative property of $S_q^{\equiv}(T,\mathbf{c})$, generalizing \cite[Lemma 2.1]{heath_brown_congruence_restriction}, which only considers $\mathbf{c}=\mathbf{0}$ and $N=0$, and \cite[Lemma 23]{heath_brown_circle_method}, which only considers $T=1$. The sole purpose of Lemma \ref{multiplicative_property} is to ultimately improve Lemma \ref{crude_bound_on_S_q_equiv} for use in Lemma \ref{sum_of_S_qs}. Typically, we would also need the multiplicative property to prove that $\sum_{q\geq1}q^{-n}S_q^{\equiv}(\mathbf{0})$ admits a representation as an Euler product of the appropriate $p$-adic densities; however, since the singular series will come from the terms of \eqref{series_formula_for_r_cong} with $\mathbf{c}=\mathbf{0}$, \cite[Lemma 2.2]{heath_brown_congruence_restriction} essentially already handles this (the non-homogeneous case where $N\neq0$ is a minor modification).

\begin{lemma}\label{multiplicative_property}
    Let $q=rs$ and $T=j\ell$ such that $(rj,s\ell)=1$. Let $\bar{r}$ and $\bar{j}$ be the multiplicative inverses of $r$ and $j$ modulo $s\ell$, respectively. Let $\bar{s}$ and $\bar{\ell}$ be the multiplicative inverses of $s$ and $\ell$ modulo $rj$, respectively. Then
    \begin{align*}
        S_q^{\equiv}(T,\mathbf{c})=S_r^{\equiv}(j,\bar{s}\bar{\ell}\mathbf{c})S_s^{\equiv}(\ell,\bar{r}\bar{j}\mathbf{c}).
    \end{align*}
\end{lemma}

\begin{proof}
    We closely follow the proof of \cite[Lemma 2.1]{heath_brown_congruence_restriction}. Let
    \begin{align*}
        S_q(a,T,\mathbf{c})\defeq\sum_{\substack{\mathbf{b}\bmod {Tq}\\ \mathbf{b}\equiv \mathbf{k}\bmod T}}e\mleft(\frac{aF_N(\mathbf{b})}{q}+\frac{\mathbf{c\cdot b}}{Tq}\mright).
    \end{align*}
    We proceed by first proving that $S_q(a,T,\mathbf{c})=S_r(\bar{s}a,j,\bar{s}\bar{\ell}\mathbf{c})S_s(\bar{r}a,\ell,\bar{r}\bar{j}\mathbf{c})$. Since the relationship is delicate, we provide a great amount of detail.
    
    Let $U=\{\mathbf{u}\in(\ZZ/rj\ZZ)^n:\mathbf{u}\equiv\mathbf{k}\bmod j\}$ and $V=\{\mathbf{v}\in(\ZZ/s\ell\ZZ)^n:\mathbf{v}\equiv \mathbf{k}\bmod\ell\}$ and $B=\{\mathbf{b}\in(\ZZ/Tq\ZZ)^n:\mathbf{b}\equiv\mathbf{k}\bmod{T}\}$. We claim that $\rho\colon U\times V\to B$ defined by $\rho(\mathbf{u},\mathbf{v})=s\ell\bar{s}\bar{\ell}\mathbf{u}+rj\bar{r}\bar{j}\mathbf{v}$ is a bijection. We first confirm that $\rho$ maps to $B$. Indeed, by the congruences $s\ell\bar{s}\bar{\ell}\equiv 1\pmod{j}$ and $s\ell\bar{s}\bar{\ell}\equiv 0\pmod{\ell}$ and $rj\bar{r}\bar{j}\equiv 0\pmod{j}$ and $rj\bar{r}\bar{j}\equiv 1\pmod{\ell}$, we see that $\rho(\mathbf{u},\mathbf{v})$ is congruent to $\mathbf{k}$ modulo $j$ and modulo $\ell$, whence it follows that $\rho(\mathbf{u},\mathbf{v})\equiv\mathbf{k}\pmod{T}$.
    
    Suppose that $s\ell\bar{s}\bar{\ell}\mathbf{u}_1+rj\bar{r}\bar{j}\mathbf{v}_1\equiv s\ell\bar{s}\bar{\ell}\mathbf{u}_2+rj\bar{r}\bar{j}\mathbf{v}_2\pmod{Tq}$. Then $s\ell\bar{s}\bar{\ell}(\mathbf{u}_1-\mathbf{u}_2)\equiv rj\bar{r}\bar{j}(\mathbf{v}_2-\mathbf{v}_1)\pmod{Tq}$, from which it follows by reducing modulo $rj$ and $s\ell$ that $\mathbf{u}_1-\mathbf{u}_2\equiv\mathbf{0}\pmod{rj}$ and $\mathbf{v}_2-\mathbf{v}_1\equiv\mathbf{0}\pmod{s\ell}$. So $\rho$ is injective. Noting that $|U|\cdot|V|=(rj/j)\cdot(s\ell/\ell)=q=Tq/T=|B|$, it follows that $\rho$ is also bijective. 
    
    Let $\mathbf{b}=\rho(\mathbf{u},\mathbf{v})$. Observe that
    \begin{align*}
        F_N(\mathbf{b})&=\sum_{i=1}^{n}\left(s\ell\bar{s}\bar{\ell}u_i+rj\bar{r}\bar{j}v_i\right)^2-N=\sum_{i=1}^{n}\left((s\ell\bar{s}\bar{\ell}u_i)^2+2Tq\bar{s}\bar{\ell}\bar{r}\bar{j}u_i v_i+(rj\bar{r}\bar{j}v_i)^2\right)-N\\
        &\equiv (s\ell\bar{s}\bar{\ell})^2\sum_{i=1}^{n}u_i^2+(rj\bar{r}\bar{j})^2\sum_{i=1}^{n}v_i^2-N\pmod{q}.
    \end{align*}
    Now since $\ell\bar{\ell}$ and $s\bar{s}$ are both congruent to $1$ modulo $r$, we have $(s\ell\bar{s}\bar{\ell})^2\equiv(s\bar{s})^2\equiv s\bar{s}\pmod{r}$. Similarly, since $j\bar{j}$ and $r\bar{r}$ are both congruent to $1$ modulo $s$, we have $(rj\bar{r}\bar{j})^2\equiv(r\bar{r})^2\equiv r\bar{r}\pmod{s}$. So since $(s\ell\bar{s}\bar{\ell})^2\equiv 0\equiv s\bar{s}\pmod{s}$ and $(rj\bar{r}\bar{j})^2\equiv 0\equiv r\bar{r}\pmod{r}$ and $(r,s)=1$, we have $(s\ell\bar{s}\bar{\ell})^2\equiv s\bar{s}\pmod{q}$ and $(rj\bar{r}\bar{j})^2\equiv r\bar{r}\pmod{q}$.

    Furthermore, we have $s\bar{s}+r\bar{r}\equiv s\bar{s}\equiv 1\pmod{r}$ and $s\bar{s}+r\bar{r}\equiv r\bar{r}\equiv 1\pmod{s}$. So $s\bar{s}+r\bar{r}\equiv 1\pmod{rs}$. This, in particular, implies that $N\equiv s\bar{s}N+r\bar{r}N\pmod{q}$. Therefore,
    \begin{align*}
        F_N(\mathbf{b})\equiv s\bar{s}\sum_{i=1}^{n}u_i^2+r\bar{r}\sum_{i=1}^{n}v_i^2-N\equiv s\bar{s}F_N(\mathbf{u})+r\bar{r}F_N(\mathbf{v})\pmod{q},
    \end{align*}
    giving that
    \begin{align*}
        e\mleft(\frac{aF_N(\mathbf{b})}{q}\mright)=e\mleft(\frac{\bar{s}aF_N(\mathbf{u})}{r}\mright) e\mleft(\frac{\bar{r}aF_N(\mathbf{v})}{s}\mright).
    \end{align*}
    Lastly, we have
    \begin{align*}
        e\mleft(\frac{\mathbf{c\cdot\mathbf{b}}}{Tq}\mright)=e\mleft(\frac{(\bar{s}\bar{\ell}\mathbf{c})\cdot \mathbf{u}}{rj}\mright) e\mleft(\frac{(\bar{r}\bar{j}\mathbf{c})\cdot \mathbf{v}}{s\ell}\mright).
    \end{align*}
    Thus, $S_q(a,T,\mathbf{c})$ equals
    \begin{align*}
        \sum_{(\mathbf{u},\mathbf{v})\in U\times V}e\mleft(\frac{\bar{s}aF_N(\mathbf{u})}{r}+\frac{(\bar{s}\bar{\ell}\mathbf{c})\cdot\mathbf{u}}{rj}\mright)e\mleft(\frac{\bar{r}aF_N(\mathbf{v})}{s}+\frac{(\bar{r}\bar{j}\mathbf{c})\cdot\mathbf{v}}{s\ell}\mright)=S_r(\bar{s}a,j,\bar{s}\bar{\ell}\mathbf{c})S_s(\bar{r}a,\ell,\bar{r}\bar{j}\mathbf{c}).
    \end{align*}
    We now proceed to proving the multiplicative property for $S_q^{\equiv}(T,\mathbf{c})$. First, we establish that $\eta\colon(\ZZ/r\ZZ)^{\times}\times(\ZZ/s\ZZ)^{\times}\to(\ZZ/q\ZZ)^{\times}$ defined by $\eta(t_1,t_2)=st_1+rt_2$ is a bijection. Confirming that $\eta$ indeed maps to $(\ZZ/q\ZZ)^{\times}$, by the congruences $st_1+rt_2\equiv st_1\pmod{r}$ and $st_1+rt_2\equiv rt_2\pmod{s}$ and the fact that $(st_1,r)=(rt_2,s)=1$, we have that $(st_1+rt_2,q)=1$. 
    
    Suppose that $st_1+rt_2\equiv st_1'+rt_2'\pmod{q}$. Then $s(t_1-t_1')\equiv r(t_2'-t_2)\pmod{q}$, from which it follows by reducing modulo $r$ and $s$ that $t_1-t_1'\equiv 0\pmod{r}$ and $t_2'-t_2\equiv{0}\pmod{s}$. So $\eta$ is injective. Since $|(\ZZ/r\ZZ)^{\times}\times(\ZZ/s\ZZ)^{\times}|=|(\ZZ/q\ZZ)^{\times}|$, it follows that $\eta$ is also bijective. So, by the multiplicative property for $S_q(a,T,\mathbf{c})$ together with the congruences $\bar{s}(st_1+rt_2)\equiv t_1\pmod{r}$ and $\bar{r}(st_1+rt_2)\equiv t_2\pmod{s}$,
    \begin{align*}
        S_q^{\equiv}(T,\mathbf{c})&=\sum_{\substack{0\leq t_1<r\\ (t_1,r)=1}}\sum_{\substack{0\leq t_2<s\\ (t_2,s)=1}} S_q(st_1+rt_2,T,\mathbf{c})\\
        &=\sum_{\substack{0\leq t_1<r\\ (t_1,r)=1}}\sum_{\substack{0\leq t_2<s\\ (t_2,s)=1}} S_r(\bar{s}(st_1+rt_2),j,\bar{s}\bar{\ell}\mathbf{c})S_s(\bar{r}(st_1+rt_2),\ell,\bar{r}\bar{j}\mathbf{c})\\
        &=\sum_{\substack{0\leq t_1<r\\ (t_1,r)=1}}S_r(t_1,j,\bar{s}\bar{\ell}\mathbf{c})\sum_{\substack{0\leq t_2<s\\ (t_2,s)=1}}S_s(t_2,\ell,\bar{r}\bar{j}\mathbf{c})=S_r^{\equiv}(j,\bar{s}\bar{\ell}\mathbf{c})S_s^{\equiv}(\ell,\bar{r}\bar{j}\mathbf{c}),
    \end{align*}
    giving the multiplicative property desired.
\end{proof}

By Lemma \ref{multiplicative_property}, we can write $S_q^{\equiv}$ as a product of the form $\prod_{p\mid q} S_{p^{v_p(q)}}^{\equiv}$. So if $q$ is square-free, the problem of bounding $S_q^{\equiv}$ is reduced to bounding $S_p^{\equiv}$ for $p$ a prime divisor of $q$. We will see later that this reduction is quite useful. We show in the following lemma that bounding $S_p^{\equiv}$ can be reduced to bounding Kloosterman sums (when $n$ is even) and Sali\'e sums (when $n$ is odd), for which the Weil bound suffices. Since \cite[Lemma 26]{heath_brown_circle_method} does not handle a congruence restriction, the only stipulation on $p$ (for the form $F_0$) is that $p$ must be odd. However, to reduce the evaluation of $S_p^{\equiv}$ to a Kloosterman or Sali\'e sum, we will need to complete the square modulo $p$, for which we require our modulus $T$ to have a multiplicative inverse modulo $p$. Hence, our condition that $p\nmid T$.

\begin{lemma}\label{prime_bound_on_S}
    Let $p$ be an odd prime not dividing $T\in\ZZ^+$. Then
    \begin{align*}
        |S_p^{\equiv}(T,\mathbf{c})|\leq 2p^{(n+1)/2}(N,\mathbf{c},p)^{1/2}.
    \end{align*}
    We also note the trivial bound $|S^{\equiv}_p(T,\mathbf{c})|\leq (p-1)p^n$, which holds for all primes $p$.
\end{lemma}

\begin{proof}
    We closely follow the proof of \cite[Lemma 26]{heath_brown_circle_method}. We have
    \begin{align*}
        S_p^{\equiv}(T,\mathbf{c})&=\sum_{a=1}^{p-1}\sum_{\substack{\mathbf{b}\bmod Tp\\ \mathbf{b}\equiv\mathbf{k}\bmod T}}e\mleft(\frac{aF_N(\mathbf{b})}{p}+\frac{\mathbf{c\cdot b}}{Tp}\mright)\\
        &=\sum_{a=1}^{p-1}\sum_{\substack{\mathbf{b}\bmod p}}e\mleft(\frac{aF_N(T\mathbf{b}+\mathbf{k})}{p}+\frac{\mathbf{c\cdot (Tb+\mathbf{k})}}{Tp}\mright)\\
        &=e\mleft(\frac{\mathbf{c\cdot k}}{Tp}\mright) \sum_{a=1}^{p-1}\sum_{\mathbf{b}\bmod p} e\mleft(\frac{aF_N(T\mathbf{b}+\mathbf{k})+\mathbf{c\cdot b}}{p}\mright)\\
        &=e\mleft(\frac{\mathbf{c\cdot k}}{Tp}\mright) \sum_{a=1}^{p-1}\sum_{\mathbf{b}\bmod p}e\mleft(\frac{a\left(\sum_{i=1}^{n}(T^2b_i^2+2Tk_i b_i)+|\mathbf{k}|^2-N\right)+\sum_{i=1}^{n}c_i b_i}{p}\mright)\\
        &=e\mleft(\frac{\mathbf{c\cdot k}}{Tp}\mright) \sum_{a=1}^{p-1}e\mleft(\frac{a(|\mathbf{k}|^2-N)}{p}\mright)\sum_{\mathbf{b}\bmod p}e\mleft(\frac{\sum_{i=1}^{n}(aT^2b_i^2+(2Tk_i a+c_i)b_i)}{p}\mright)\\
        &=e\mleft(\frac{\mathbf{c\cdot k}}{Tp}\mright) \sum_{a=1}^{p-1}e\mleft(\frac{a(|\mathbf{k}|^2-N)}{p}\mright)\prod_{i=1}^{n}\sum_{{b}\bmod p}e\mleft(\frac{aT^2b^2+(2Tk_i a+c_i)b}{p}\mright).
    \end{align*}
    Let $A_a=aT^2$ and $B_{i,a}=2Tk_i a+c_i$. We henceforth use an overline upon a quantity to denote its multiplicative inverse modulo $p$. Since $p\nmid T$, we complete the square modulo $p$ by
    \begin{align*}
        aT^2b^2+(2Tk_i a+c_i)b\equiv A_a(b+\overline{2A_a}B_{i,a})^2-\overline{4A_a}B_{i,a}^2\pmod{p}.
    \end{align*}
    Therefore
    \begin{align*}
        \sum_{b\bmod p}e\mleft(\frac{aT^2b^2+(2Tk_i a+c_i)b}{p}\mright)=e\mleft(-\frac{\overline{4A_a}B_{i,a}^2}{p}\mright)\sum_{b\bmod p}e\mleft(\frac{A_a(b+\overline{2A_a}B_{i,a})^2}{p}\mright)\\
        =e\mleft(-\frac{\overline{4A_a}B_{i,a}^2}{p}\mright)\sum_{b\bmod p}e\mleft(\frac{A_a b^2}{p}\mright)=e\mleft(-\frac{\overline{4A_a}B_{i,a}^2}{p}\mright)\left(\frac{A_a}{p}\right)\tau_p=e\mleft(-\frac{\overline{4A_a}B_{i,a}^2}{p}\mright)\left(\frac{a}{p}\right)\tau_p,
    \end{align*}
    where $\tau_p=i^{(p-1)^2/4}\sqrt{p}$. By considering the equality
    \begin{align*}
        \prod_{i=1}^{n}e\mleft(-\frac{\overline{4A_a}B_{i,a}^2}{p}\mright)=e\mleft(\frac{-\overline{T}\mathbf{c\cdot k}}{p}\mright)e\mleft(\frac{-a|\mathbf{k}|^2}{p}\mright)e\mleft(\frac{-\overline{a}(\overline{4T^2}|\mathbf{c}|^2)}{p}\mright),
    \end{align*}
    we obtain
    \begin{align*}
        S^{\equiv}_p(T,\mathbf{c})=e\mleft(\frac{(T^{-1}-\overline{T})(\mathbf{c\cdot k})}{p}\mright)\tau_p^{n} \sum_{a=1}^{p-1}\left(\frac{a}{p}\right)^n e\mleft(\frac{a(-N)+\overline{a}(-\overline{4T^2}|\mathbf{c}|^2)}{p}\mright).
    \end{align*}
    The remaining sum is then a Kloosterman sum when $n$ is even, and is a Sali\'e sum when $n$ is odd. Using Weil's bound \cite{weil_bound} and Sali\'e's evaluation \cite{Salie}, we finally have
    \begin{align*}
        |S_p^{\equiv}(T,\mathbf{c})|\leq p^{n/2}\cdot 2p^{1/2}(-N,-\overline{4T^2}|\mathbf{c}|^2,p)^{1/2}=2p^{(n+1)/2}(N,\mathbf{c},p)^{1/2},
    \end{align*}
    as desired.
\end{proof}

We now arrive at the punchline set up by Lemmas \ref{multiplicative_property} and \ref{prime_bound_on_S}: writing $S_q^{\equiv}$ as $S_u^{\equiv}\cdot S_v^{\equiv}$ with $u$ square-free and $v$ square-full, we can apply Lemma \ref{crude_bound_on_S_q_equiv} to $S_v^{\equiv}$ and Lemma \ref{prime_bound_on_S} to almost all factors $S_p^{\equiv}$ in the prime product decomposition of $S_u^{\equiv}$ to yield a refined estimate for $S_q^{\equiv}$. Then, parameterizing $q\leq X$ by $u$ square-free and $v$ square-full yields an almost-square-root improvement to our estimate of $\sum_{q\leq X} |S_{q}^{\equiv}(\mathbf{c})|$ over simply applying Lemma \ref{crude_bound_on_S_q_equiv}. This is the biggest hurdle toward obtaining an error as good as $O(N^{(n-1)/4+\varepsilon})$.

We note that approximation \eqref{c_0_series_estimate} also holds for $0<|\mathbf{c}|\leq P$, but we are only interested in the $\mathbf{c}=\mathbf{0}$ case, since that corresponds to the leading-order term of $r_n^{\equiv}(N)$.

\begin{lemma}\label{sum_of_S_qs}
   Let $|\mathbf{c}|\leq P$. Then
    \begin{align}\label{bound_on_sum_kloosterman}
        \sum_{q\leq X} |S_{q}^{\equiv}(\mathbf{c})|\ll X^{(3+n)/2+\varepsilon} P^{\varepsilon}.
    \end{align}
    Consequently,
    \begin{align}\label{c_0_series_estimate}
        \sum_{q\leq X}(Mq)^{-n}S_q^{\equiv}(\mathbf{0})=\sum_{q\geq1} (Mq)^{-n}S_q^{\equiv}(\mathbf{0})+O(X^{(3-n)/2+\varepsilon}P^{\varepsilon}),
    \end{align}
    and the infinite sum converges absolutely.
\end{lemma}

\begin{proof}
    This result and its proof are direct modifications of \cite[Lemma 28]{heath_brown_circle_method}. We begin by refining Lemma \ref{crude_bound_on_S_q_equiv} to obtain an almost-square-root improvement by leveraging the multiplicative property exhibited by Lemma \ref{multiplicative_property}. For $q\in\ZZ^+$, let $u=\prod_{p:v_p(q)=1} p$ and $v=\prod_{p:v_p(q)\geq2}p^{v_p(q)}$ be the square-free and square-full parts of $q=uv$. Letting $j=\prod_{p:v_p(q)=1}p^{v_p(M)}$ and $\ell=\prod_{p:v_p(q)\geq2}p^{v_p(M)}$ so that $M=j\ell$, by Lemmas \ref{multiplicative_property} and \ref{crude_bound_on_S_q_equiv} we have
    \begin{align*}
        S_q^{\equiv}(\mathbf{c})=S_{u}^{\equiv}(j,\bar{v}\bar{\ell}\mathbf{c}) S_v^{\equiv}(\ell,\bar{u}\bar{j}\mathbf{c})\ll v^{1+n/2} \left|S_u^{\equiv}(j,\bar{v}\bar{\ell}\mathbf{c})\right|.
    \end{align*}
    Since $u$ is square-free, by Lemma \ref{multiplicative_property},
    \begin{align*}
        \left|S_u^{\equiv}(j,\bar{v}\bar{\ell}\mathbf{c})\right|=\prod_{p\mid u}\left|S_p^{\equiv}(p^{v_p(j)},B_p\mathbf{c})\right|,
    \end{align*}
    where each $B_p$ is an integer coprime to $p$. So by Lemma \ref{prime_bound_on_S}, $\left|S_u^{\equiv}(j,\bar{v}\bar{\ell}\mathbf{c})\right|$ is at most
    \begin{align*}
        2^n\prod_{\substack{p\mid u\\ p\mid j}}(p-1)p^n\prod_{\substack{p\mid u\\ p\nmid j}}2p^{\frac{n+1}{2}}(p,N,B_p\mathbf{c})^{1/2}
        \ll \prod_{\substack{p\mid u}}2p^{\frac{n+1}{2}}(p,N,B_p\mathbf{c})^{1/2}= 2^{\omega(u)}u^{\frac{n+1}{2}}(u,N,\mathbf{c})^{1/2}.
    \end{align*}
    Following the rest of the proof of \cite[Lemma 28]{heath_brown_circle_method}, mutatis mutandis, yields \eqref{bound_on_sum_kloosterman}. Finally, we see by summation by parts that \eqref{bound_on_sum_kloosterman} implies approximation \eqref{c_0_series_estimate} and the absolute convergence of the infinite series therein. 
\end{proof}

\subsection{Relating $r_n^{\equiv}(N)$ to $r_n^{\square}(N)$}

By our estimates on $I_q^{\equiv}(\mathbf{c})$ and $S_q^{\equiv}(\mathbf{c})$, the error of $O(N^{(n-1)/4+\varepsilon})$ follows by \cite[pg.200]{heath_brown_circle_method}, mutatis mutandis. Thus, we turn to considering the leading order term of $r_n^{\equiv}(N)$. Lemma \ref{I_q_0_cong_est} handles the contribution from $I_q^{\equiv}(\mathbf{0})$, so it remains to note the Euler product representation of $p$-adic densities for $\sum_{q\geq1}(Mq)^{-n}S_q^{\equiv}(\mathbf{0})$. This follows upon using \cite[Lemma 2.2]{heath_brown_congruence_restriction} and replacing $F(\mathbf{x})$ with $F_N(\mathbf{x})$ since $S_q^{\equiv}(\mathbf{0})$ is already directly handled by \cite{heath_brown_congruence_restriction} (with the notation $S(q,M,\mathbf{k})$ instead). Furthermore, the convergence of the Euler product is assured by Lemma \ref{sum_of_S_qs} by asserting that $\sum_{q\geq1} (Mq)^{-n}S_q^{\equiv}(\mathbf{0})$ is indeed absolutely convergent.

\begin{lemma}[{\cite[pgs.152--154]{heath_brown_congruence_restriction}}]
    We have
    \begin{align*}
        \sum_{q\geq1} (Mq)^{-n}S_q^{\equiv}(\mathbf{0})=M^{-n}\prod_{p}\sigma_p^{\equiv}(N,n),
    \end{align*}
    where $\sigma_p^{\equiv}(N,n)$ is the $p$-adic density
    \begin{align*}
        \sigma_p^{\equiv}(N,n)\defeq\lim_{d\to\infty}p^{-d(n-1)}{\#\left\{\mathbf{x}\in\left(\ZZ/p^{d+v_p(M)}\ZZ\right)^n:\sum_{i=1}^{n}x_i^2\equiv N\bmod p^d,~x_i\equiv k_i\bmod p^{v_p(M)}\right\}}.
    \end{align*}
\end{lemma}

We are now able to prove the theorem that is this section's purpose, successfully modifying Heath-Brown's circle method to obtain an explicit asymptotic relation between $r_n^{\equiv}(N)$ and $r_n^{\square}(N)$ in terms of their $p$-adic densities, which we evaluate in the next section. The $p$-adic density $\delta_p^{\square}(N,n)$ of $r_n^{\square}(N)$ is
\begin{align}\label{square_p_adic_density}
    \delta_p^{\square}(N,n)\defeq\lim_{d\to\infty}p^{-d(n-1)}{\#\left\{\mathbf{x}\in\left(\ZZ/p^{d}\ZZ\right)^n:\sum_{i=1}^{n}x_i^2\equiv N\bmod p^d\right\}}.
\end{align}
With this notation, we may now state the theorem sought.

\begin{theorem}\label{modification_theorem}
    Let $n\geq4$ and $\mathbf{k}\in\ZZ^n$ and $M\in\ZZ^+$. Then the number $r_n^{\equiv}(N)$ of solutions over $\ZZ^n$ to $\sum_{i=1}^{n}x_i^2=N$ such that each $x_i\equiv k_i\bmod{M}$ satisfies
    \begin{align*}
        r_n^{\equiv}(N)=M^{-n}\left(\prod_{p\mid M}\frac{\sigma_p^{\equiv}(N,n)}{\delta_p^{\square}(N,n)}\right)r_n^{\square}(N)+O(N^{(n-1)/4+\varepsilon}).
    \end{align*}
\end{theorem}

\begin{proof}
    We have shown that $I_q^{\equiv}(\mathbf{c})=0$ holds for $q\geq P$ as it does for $I_q(\mathbf{c})$, and have sufficiently modified or generalized \cite[Lemmas 13, 16, 19, 22, 25, 23, 26, 28]{heath_brown_circle_method} and the relevant portion of \cite[Lemma 31]{heath_brown_circle_method}. Thus, following the proof in \cite[pg.200]{heath_brown_circle_method}, mutatis mutandis, yields
    \begin{align*}
        r_n^{\equiv}(N)&=\sigma_{\infty}(G)M^{-n}\left(\prod_{p}\sigma_p^{\equiv}(N,n)\right)N^{n/2-1}+O(N^{(n-1)/4+\varepsilon})\\
        &=\left(M^{-n}\prod_{p\mid M}\sigma_p^{\equiv}(N,n)\right)\left(\sigma_{\infty}(G)\prod_{p\nmid M}\delta^{\square}_{p}(N,n)N^{n/2-1}\right)+O(N^{(n-1)/4+\varepsilon})
    \end{align*}
    upon noting that $\sigma_{p}^{\equiv}(N,n)=\delta_p^{\square}(N,n)$ when $p\nmid M$. By \cite[Corollary 1]{heath_brown_circle_method}, we see that $\sigma_{\infty}(G)\prod_{p}\delta^{\square}_{p}(N,n)N^{n/2-1}=r_n^{\square}(N)+O(N^{(n-1)/4+\varepsilon})$, yielding the desired asymptotic relation.
\end{proof}

\section{Evaluating the $p$-adic Densities}\label{evaluating_the_p_adic_densities}

Having proven Theorem \ref{modification_theorem}, we finally turn to evaluating the square and $k$-gonal $p$-adic densities. The square $p$-adic density $\delta_p^{\square}(N,n)$ is as defined in \eqref{square_p_adic_density}, and the $k$-gonal $p$-adic density $\delta_p^{\ksymbol{k}}(N,n)$ is
    \begin{align*}
        \delta_p^{\ksymbol{k}}(N,n)\defeq\lim_{d\to\infty}p^{-d(n-1)}{\#\left\{\mathbf{x}\in\left(\ZZ/p^{d+v_p(M)}\ZZ\right)^n:\sum_{i=1}^{n}x_i^2\equiv X_N\bmod p^d,~x_i\equiv k\bmod p^{v_p(M)}\right\}},
    \end{align*}
where $X_N=8(k-2)N+n(k-4)^2$ as before and $M=2(k-2)$. With this notation, $\mathbf{k}=(k,\dots,k)$ gives $\delta_{p}^{\ksymbol{k}}(N,n)=\sigma_p^{\equiv}(X_N,n)$. Recall that we care about $\sigma_p^{\equiv}(X_N,n)$ because of Theorem \ref{modification_theorem} and $r_n^{\ksymbol{k}}(N)=r_n^{\equiv}(X_N)$.

We evaluate the $p$-adic densities $\delta_p^{\ksymbol{k}}(N,n)$ and $\delta_p^{\square}(N,n)$ by the method of exponential sums: for $\delta_p^{\square}(N,n)$, we have
\begin{align*}
    \delta_p^{\square}(N,n)=\lim_{d\to\infty}p^{-nd}\sum_{a\bmod p^d}\left(\sum_{r\bmod p^d}e\mleft(\frac{ar^2}{p^d}\mright)\right)^n e\mleft(\frac{-aN}{p^d}\mright),
\end{align*}
and for $\delta_p^{\ksymbol{k}}(N,n)$, letting $v=v_p(2(k-2))$, we have
\begin{align*}
     \delta_p^{\ksymbol{k}}(N,n)&=\lim_{d\to\infty}p^{-d(n-1)}\left(p^{-d}\sum_{a\bmod p^d}\left(\sum_{\substack{r\bmod  p^{d+v}\\ r\equiv k\bmod{p^{v}}}}e\mleft(\frac{ar^2}{p^d}\mright)\right)^n e\mleft(\frac{-aX_N}{p^d}\mright)\right)\\
     &=\lim_{d\to\infty}p^{-nd}\sum_{a\bmod p^d}\left(\sum_{r\bmod p^d}e\mleft(\frac{a(rp^{v}+k)^2}{p^d}\mright)\right)^n e\mleft(\frac{-aX_N}{p^d}\mright).
\end{align*}

We find that the computation of $\delta_p^{\ksymbol{k}}(N,n)$ is largely quite pleasant since the exponential sum $\sum_{r\bmod p^d}e\mleft({a(rp^{v}+k)^2}/{p^d}\mright)$ has a simple evaluation, except in the case where $p=2$ and $k\equiv0\pmod{4}$. Indeed, aside from this exceptional case, for this computation, one really only needs some standard facts about generalized quadratic Gauss sums. Remarkably, we find that the evaluation of the polygonal $p$-adic densities is uniform in $N$ and $n$ when $p\mid k-2$ is odd, and is uniform in $N$ and $n$ for $p=2$ if $k\not\equiv0\pmod{4}$. In contrast to the polygonal $p$-adic densities, the computation of $\delta_p^{\square}(N,n)$ is incredibly tedious due to the more complicated Gauss' evaluation of $\sum_{r\bmod p^d}e\mleft({Ar^2}/{p^d}\mright)$, as recalled in \eqref{gauss_formula_berndt}.

We will use the characteristic function notation $\mathbf{1}_{S}$ throughout this section in our results and proofs thereof: for some statement $S$, we set $\mathbf{1}_S=1$ if $S$ is true, and $\mathbf{1}_{S}=0$ if $S$ is false.

\subsection{The square $p$-adic densities}\label{subsection_square_densities}

We first compute the square $p$-adic densities, which, once done, suffice for Theorem \ref{more_general_theorem} once paired with Theorem \ref{modification_theorem}. The closed formulas for $\delta_p^{\square}(N,n)$ are complicated, and their proofs are incredibly tedious (especially for $p=2$).

For $(A,C)=1$, we recall (see \cite[Chapter 1]{berndt1998gauss}, for example) the evaluation of the quadratic Gauss sum
\begin{align}\label{gauss_formula_berndt}
    \sum_{r\bmod C}e\mleft(\frac{Ar^2}{C}\mright)=\begin{cases}
        (1+i)\varepsilon_A^{-1}\sqrt{C}\left(\frac{C}{A}\right) & \text{if~}c\equiv0\pmod{4},\\
        \varepsilon_C \sqrt{C}\left(\frac{A}{C}\right) & \text{if~}c\equiv 1\pmod{2},\\
        0 & \text{if~}c\equiv2\pmod{4},
    \end{cases}
\end{align}
where $\left(\frac{\cdot}{\cdot}\right)$ is the Jacobi symbol and $\varepsilon_m\defeq i^{(m-1)^2/4}$. We first evaluate $\delta_p^{\square}(N,n)$ for odd $p$, which we can state by a single closed form, albeit somewhat unnaturally.

\begin{prop}\label{odd_adic_square}
    Let $p$ be an odd prime and $v=v_p(N)$ and $N_p=p^{-v}N$ be the $p$-free part of $N$. Then $\delta_p^{\square}(N,n)=1+p^{-n/2}f_p(N,n)$, where
    \begin{align*}
        f_p(N,n)=(p&-1)\left(\cos\mleft(\frac{\pi n}{2}\mright)(-1)^{n(p+1)/4}\frac{1-p^{2\lceil v/2\rceil(1-n/2)}}{1-p^{2-n}}+p^{1-n/2}\frac{1-p^{2\lceil (v-1)/2\rceil(1-n/2)}}{1-p^{2-n}}\right)\\
        &-(-1)^{\lfloor n/2\rfloor(v+1)(p-1)/2}\left(-\left(\frac{N_p}{p}\right)\right)^{n(v+1)}p^{v(1-n/2)+(1-(-1)^n)(1+(-1)^v)/8}.
    \end{align*}
\end{prop}

\begin{proof}
    Let $a\neq0$ be modulo $p^d$ and $a_p=ap^{-v_p(a)}$ be its $p$-free part. We then have
    \begin{align}\label{quad_gauss_sum_eval_odd}
        \sum_{r\bmod p^d}e\mleft(\frac{ar^2}{p^d}\mright)=p^{v_p(a)}\sum_{t\bmod p^{d-v_p(a)}}e\mleft(\frac{a_p t^2}{p^{d-v_p(a)}}\mright)=p^{v_p(a)}\varepsilon_{p^{d-v_p(a)}}\left(\frac{a_p}{p^{d-v_p(a)}}\right) p^{(d-v_p(a))/2},
    \end{align}
    where the first equality follows from the substitution $r\mapsto t+bp^{d-v_p(a)}$ with $b,t$ modulo $p^{v_p(a)}$ and the second equality follows from Gauss' formula \eqref{gauss_formula_berndt}. For a function $f$, we note that we can partition the sum $\sum_{a\bmod p^d}f(a)$ via the $p$-adic valuation of the variable $a$ by
    \begin{align*}
        \sum_{a\bmod p^d}f(a)=f(0)+\sum_{j=1}^{p-1}\sum_{c=1}^{d}\sum_{t\bmod p^{d-c}} f(jp^{c-1}+tp^c).
    \end{align*}
    Using this formula on the exponential sum formula for $\delta_p^{\square}(N,n)$, we have that $\delta_p^{\square}(N,n)$ is the limit as $d\to\infty$ of
    \begin{align}\label{decomp_of_delta_p_square_sum}
        1+p^{-nd}\sum_{j=1}^{p-1}\sum_{c=1}^{d}\sum_{t\bmod p^{d-c}}\left(\sum_{r\bmod p^d}e\mleft(\frac{(jp^{c-1}+tp^c)r^2}{p^d}\mright)\right)^n e\mleft(-\frac{(jp^{c-1}+tp^c)N}{p^d}\mright).
    \end{align}
    By \eqref{quad_gauss_sum_eval_odd},
    \begin{align*}
        \sum_{r\bmod p^d}e\mleft(\frac{(jp^{c-1}+tp^c)r^2}{p^d}\mright)=p^{c-1}\varepsilon_{p^{d-c+1}}\left(\frac{j+tp}{p^{d-c+1}}\right)p^{(d-c+1)/2}=\left(\varepsilon_p\left(\frac{j}{p}\right)\right)^{\mathbf{1}_{c\equiv d\bmod 2}}p^{(d+c-1)/2}.
    \end{align*}
    So the sum over $t$ modulo $p^{d-c}$ in \eqref{decomp_of_delta_p_square_sum} equals
    \begin{align*}
        e\mleft(\frac{-jN}{p^{d-c+1}}\mright)\left(\left(\varepsilon_p\left(\frac{j}{p}\right)\right)^{\mathbf{1}_{c\equiv d\bmod 2}}p^{(d+c-1)/2}\right)^n\sum_{t\bmod p^{d-c}}e\mleft(\frac{-tN}{p^{d-c}}\mright)\\
        =p^{n(d+c-1)/2+d-c}e\mleft(\frac{-jN}{p^{d-c+1}}\mright)\left(\varepsilon_p\left(\frac{j}{p}\right)\right)^{n(\mathbf{1}_{c\equiv d\bmod 2})}\mathbf{1}_{v_p(N)\geq d-c}.
    \end{align*}
    So the sum over $1\leq c\leq d$ in \eqref{decomp_of_delta_p_square_sum} equals
    \begin{align}\label{c_sums}
        &~~~~~\sum_{c=d-v_p(N)}^{d}p^{n(d+c-1)/2+d-c}e\mleft(\frac{-jN}{p^{d-c+1}}\mright)\left(\varepsilon_p\left(\frac{j}{p}\right)\right)^{n(\mathbf{1}_{c\equiv d\bmod 2})}\nonumber\\
        &=p^{nd-n/2}\sum_{c=0}^{v_p(N)} p^{c(1-n/2)}e\mleft(\frac{-jN}{p^{c+1}}\mright)\left(\varepsilon_p\left(\frac{j}{p}\right)\right)^{n(\mathbf{1}_{c \text{~even}})}\nonumber\\
        &=p^{nd-n/2}\left(\sum_{c=0}^{\lfloor v_p(N)/2\rfloor}p^{2c(1-n/2)}\left(\varepsilon_p\left(\frac{j}{p}\right)\right)^n e\mleft(\frac{-jN}{p^{2c+1}}\mright)+\sum_{c=0}^{\lfloor (v_p(N)-1)/2\rfloor}p^{(2c+1)(1-n/2)}e\mleft(\frac{-jN}{p^{2c+2}}\mright)\right).
    \end{align}
    Passing the sum over $1\leq j\leq p-1$ inside, we are interested in evaluating
    \begin{align}\label{def_first_c_sum}
        \sum_{c=0}^{\lfloor v_p(N)/2\rfloor}p^{2c(1-n/2)}\sum_{j=1}^{p-1}\left(\varepsilon_p\left(\frac{j}{p}\right)\right)^n e\mleft(\frac{-jN}{p^{2c+1}}\mright)
    \end{align}
    and
    \begin{align}\label{def_second_c_sum}
        \sum_{c=0}^{\lfloor (v_p(N)-1)/2\rfloor}p^{(2c+1)(1-n/2)}\sum_{j=1}^{p-1}e\mleft(\frac{-jN}{p^{2c+2}}\mright).
    \end{align}
    We begin by handling \eqref{def_first_c_sum}. If $c<v_p(N)/2$, then
    \begin{align}\label{first_c_sum_small}
        \sum_{j=1}^{p-1}\left(\varepsilon_p\left(\frac{j}{p}\right)\right)^n e\mleft(\frac{-jN}{p^{2c+1}}\mright)=\varepsilon_p^n\sum_{j=1}^{p-1}\left(\frac{j}{p}\right)^n=\varepsilon_p^n(p-1)\mathbf{1}_{n\text{~even}}.
    \end{align}
    If $c=v_p(N)/2$, then
    \begin{align}\label{first_c_sum_edge}
        \sum_{j=1}^{p-1}\left(\varepsilon_p\left(\frac{j}{p}\right)\right)^n e\mleft(\frac{-jN}{p^{2c+1}}\mright)=\varepsilon_p^n\sum_{j=1}^{p-1}\left(\frac{j}{p}\right)^n e\mleft(\frac{-jN_p}{p}\mright)=\begin{cases}
            -\varepsilon_p^n & n\text{~even,}\\
            \varepsilon_p^{n}\left(\frac{-N_p}{p}\right)\varepsilon_p p^{1/2} & n\text{~odd,}
        \end{cases}
    \end{align}
    by the standard Gauss sum evaluation. We note the simple identity $\varepsilon_p^{n}\left(\frac{-N_p}{p}\right)\varepsilon_p p^{1/2}=(-1)^{(n+1)(p-1)/4}\left(\frac{-N_p}{p}\right)p^{1/2}=(-1)^{(n-1)(p-1)/4}\left(\frac{N_p}{p}\right)p^{1/2}$ for $n$ odd, which we will use instead as it allows us to nicely unify our final expressions that depend on the parity on $n$.


    Suppose that $v_p(N)$ is even. Then if $n$ is even, by \eqref{first_c_sum_small} and \eqref{first_c_sum_edge},
    \begin{align*}
        \sum_{j=1}^{p-1}\sum_{c=0}^{\lfloor v_p(N)/2\rfloor}p^{2c(1-n/2)}\left(\varepsilon_p \left(\frac{j}{p}\right)\right)^n e\mleft(\frac{-jN}{p^{2c+1}}\mright)
        =\sum_{c=0}^{v_p(N)/2-1}p^{2c(1-n/2)}\varepsilon_p^n(p-1)-\varepsilon_p^n p^{v_p(N)(1-n/2)}\\
        =\varepsilon_p^n (p-1)\frac{1-p^{v_p(N)(1-n/2)}}{1-p^{2-n}}-(-1)^{{n(p-1)}/{4}} p^{v_p(N)(1-n/2)}.
    \end{align*}
    If $n$ is odd, then by \eqref{first_c_sum_small} and \eqref{first_c_sum_edge},
    \begin{align*}
        \sum_{j=1}^{p-1}\sum_{c=0}^{\lfloor v_p(N)/2\rfloor}p^{2c(1-n/2)}\left(\varepsilon_p \left(\frac{j}{p}\right)\right)^n e\mleft(\frac{-jN}{p^{2c+1}}\mright)
        =(-1)^{(n-1)(p-1)/4}\left(\frac{N_p}{p}\right)p^{v_p(N)(1-n/2)+1/2}.
    \end{align*}
    Now suppose that $v_p(N)$ is odd. Then
    \begin{align*}
        \sum_{j=1}^{p-1}\sum_{c=0}^{\lfloor v_p(N)/2\rfloor}p^{2c(1-n/2)}\left(\varepsilon_p \left(\frac{j}{p}\right)\right)^n e\mleft(\frac{-jN}{p^{2c+1}}\mright)
        =\varepsilon_p^n(p-1)\frac{1-p^{(v_p(N)+1)(1-n/2)}}{1-p^{2-n}}\mathbf{1}_{n~\text{even}},
    \end{align*}
    by \eqref{first_c_sum_small} since we avoid the edge case where $c=v_p(N)/2$. This completes the evaluation of \eqref{def_first_c_sum}.

    We now handle the evaluation of \eqref{def_second_c_sum}. If $c<(v_p(N)-1)/2$, then $\sum_{j=1}^{p-1}e(-jN/p^{2c+2})=p-1$, and if $c=(v_p(N)-1)/2$, then $\sum_{j=1}^{p-1}e(-jN/p^{2c+2})=-1$. Thus, if $v_p(N)$ is even, then
    \begin{align*}
        \sum_{j=1}^{p-1}\sum_{c=0}^{\lfloor(v_p(N)-1)/2\rfloor}p^{(2c+1)(1-n/2)}e\mleft(\frac{-jN}{p^{2c+2}}\mright)=(p-1)\sum_{c=0}^{(v_p(N)-2)/2} p^{(2c+1)(1-n/2)}\\
        =(p-1)p^{1-n/2}\frac{1-p^{v_p(N)(1-n/2)}}{1-p^{2-n}}.
    \end{align*}
    If $v_p(N)$ is odd, then
    \begin{align*}
        \sum_{j=1}^{p-1}\sum_{c=0}^{\lfloor(v_p(N)-1)/2\rfloor}p^{(2c+1)(1-n/2)}e\mleft(\frac{-jN}{p^{2c+2}}\mright)=(p-1)\sum_{c=0}^{(v_p(N)-1)/2-1}p^{(2c+1)(1-n/2)}-p^{v_p(N)(1-n/2)}\\
        =(p-1)p^{1-n/2}\frac{1-p^{(v_p(N)-1)(1-n/2)}}{1-p^{2-n}}-p^{v_p(N)(1-n/2)}.
    \end{align*}
    Combining our evaluations of \eqref{def_first_c_sum} and \eqref{def_second_c_sum} with \eqref{decomp_of_delta_p_square_sum} and \eqref{c_sums} gives that, for $v_p(N)$ odd, $\delta_p^{\square}(N,n)$ equals
    \begin{align*}
        1+p^{-n/2}\left((p-1)\left(\varepsilon_p^n\frac{1-p^{(v_p(N)+1)(1-n/2)}}{1-p^{2-n}}\mathbf{1}_{n~\text{even}}+p^{1-n/2}\frac{1-p^{(v_p(N)-1)(1-n/2)}}{1-p^{2-n}}\right)-p^{v_p(N)(1-n/2)}\right).
    \end{align*}
    Turning our attention to the case where $v_p(N)$ is even, note that
    \begin{align*}
        -(-1)^{\lfloor n/2\rfloor(p-1)/2}\left(-\left(\frac{N_p}{p}\right)\right)^n p^{v_p(N)(1-n/2)+(1-(-1)^n)/4}
    \end{align*}
    equals $-(-1)^{n(p-1)/4} p^{v_p(N)(1-n/2)}$ if $n$ is even, and equals $(-1)^{(n-1)(p-1)/4}\left(\frac{N_p}{p}\right)p^{v_p(N)(1-n/2)+1/2}$ if $n$ is odd. So by our evaluations of \eqref{def_first_c_sum} and \eqref{def_second_c_sum} together with \eqref{decomp_of_delta_p_square_sum} and \eqref{c_sums}, we have that if $v_p(N)$ is even, $\delta_p^{\square}(N,n)$ equals
    \begin{align*}
        1+p^{-n/2}\Bigg((p-1)\left(\varepsilon_p^n\frac{1-p^{v_p(N)(1-n/2)}}{1-p^{2-n}}\mathbf{1}_{n\text{~even}}+p^{1-n/2}\frac{1-p^{v_p(N)(1-n/2)}}{1-p^{2-n}}\right)\\
        -(-1)^{\lfloor n/2\rfloor(p-1)/2}\left(-\left(\frac{N_p}{p}\right)\right)^n p^{v_p(N)(1-n/2)+(1-(-1)^n)/4}\Bigg).
    \end{align*}
    Upon noting that $\varepsilon_p^n\mathbf{1}_{n~\text{even}}=\cos(\pi n/2)(-1)^{n(p+1)/4}$, we see that $\delta_p^{\square}(N,n)=1+p^{-n/2} f_p(N,n)$ for all $N$ and $n$, as claimed.
\end{proof}

Finally, we move to computing $\delta_2^{\square}(N,n)$, which is the most difficult of the density computations. In our evaluation of $\delta_2^{\square}(N,n)$, we encounter a geometric sum twisted by a sum of cosines, which merely forces us to break our formula into cases, but nevertheless makes the computations quite tedious. This twisted geometric sum runs over the positive integers up to $v_2(N)+2$, so we are interested in values of our sum of cosines over that range.

\begin{lemma}\label{lemma_cosine_eval}
    Let $n,N\in\ZZ^+$ and $1\leq c\leq v+2$, where $v=v_2(N)$. Set $N_2=2^{-v}N$ to be the odd part of $N$. Define $C_n(c,N)=\cos\mleft({\pi n}/{4}-{\pi N}/{2^c}\mright)+(-1)^{nc}\cos\mleft({3\pi n}/{4}-{3\pi N}/{2^c}\mright)$. Since $C_{n+4}(c,N)=-C_{n}(c,N)$, it suffices to evaluate $C_n(c,N)$ only for $n\in\{0,1,2,3\}$. For these cases, we have the evaluations
        \begin{align*}
            C_0(c,N)&=\begin{cases}
                2 & \text{if~}c<v,\\
                -2 & \text{if~}c=v,\\
                0 & \text{if~}c>v.
            \end{cases}\\
            C_1(c,N)&=\begin{cases}
                2^{1/2}\mathbf{1}_{c~\text{\normalfont odd}} & \text{if~}c<v,\\
                -2^{1/2}\mathbf{1}_{c~\text{\normalfont odd}} & \text{if~}c=v,\\
                2^{1/2}(-1)^{(N_2-1)/2}\mathbf{1}_{c\text{~\normalfont odd}} & \text{if~}c=v+1,\\
                2^{1/2}\left(\sin(N_2\pi/4)+\cos(N_2\pi/4)\right)\mathbf{1}_{c\text{\normalfont~even}} & \text{if~}c=v+2.
            \end{cases}\\
            C_2(c,N)&=\begin{cases}
                0 & \text{if~}c\neq v+1,\\
                2(-1)^{(N_2-1)/2} & \text{if~}c=v+1.
            \end{cases}\\
            C_3(c,N)&=\begin{cases}
                -2^{1/2}\mathbf{1}_{c~\text{\normalfont odd}} & \text{if~}c<v,\\
                2^{1/2}\mathbf{1}_{c~\text{\normalfont odd}} & \text{if~}c=v,\\
                2^{1/2}(-1)^{(N_2-1)/2}\mathbf{1}_{c\text{~\normalfont odd}} & \text{if~}c=v+1,\\
                2^{1/2}\left(\sin(N_2\pi/4)-\cos(N_2\pi/4)\right)\mathbf{1}_{c\text{\normalfont~even}} & \text{if~}c=v+2.
            \end{cases}
        \end{align*}
\end{lemma}

\begin{proof}
    We first note the $n=0,2$ cases since they are simpler. If $n=0$, then
    \begin{align*}
        C_n(c,N)=\cos\mleft(\frac{\pi N}{2^c}\mright)+\cos\mleft(\frac{3\pi N}{2^c}\mright)=\begin{cases}
                2 & \text{if~}1\leq c\leq v-1,\\
                -2 & \text{if~}c=v,\\
                0 & \text{if~}c>v.
            \end{cases}
    \end{align*}
    If $n=2$, then
    \begin{align*}
        C_n(c,N)=\sin\mleft(\frac{\pi N}{2^c}\mright)-\sin\mleft(\frac{3\pi N}{2^c}\mright)=\begin{cases}
                0 & \text{if~}c\neq v+1,\\
                2(-1)^{(N_2-1)/2} & \text{if~}c=v+1.
            \end{cases}
    \end{align*}
    We now consider the $n=1$ case. We then have
    \begin{align*}
        \cos\mleft(\frac{n\pi}{4}-\frac{N\pi}{2^c}\mright)=\frac{1}{\sqrt{2}} \left(\sin\mleft(\frac{N\pi}{2^c}\mright)+\cos\mleft(\frac{N\pi}{2^c}\mright)\right)
    \end{align*}
    and
    \begin{align*}
        \cos\mleft(\frac{3\pi n}{4}-\frac{3\pi N}{2^c}\mright)=\frac{1}{\sqrt{2}} \left(\sin\mleft(\frac{3N\pi}{2^c}\mright)-\cos\mleft(\frac{3N\pi}{2^c}\mright)\right)
    \end{align*}
    so that
    \begin{align*}
        C_1(c,N)=\frac{1}{\sqrt{2}}\left(\sin\mleft(\frac{N\pi}{2^c}\mright)+\cos\mleft(\frac{N\pi}{2^c}\mright)+(-1)^{c}\left(\sin\mleft(\frac{3N\pi}{2^c}\mright)-\cos\mleft(\frac{3N\pi}{2^c}\mright)\right)\right).
    \end{align*}
    So if $c\leq v$ and $c$ is even, then $C_1(c,N)=2^{-1/2}(\sin\mleft({N\pi}/{2^c}\mright)+\cos\mleft({N\pi}/{2^c}\mright)+\sin\mleft({3N\pi}/{2^c}\mright)-\cos\mleft({3N\pi}/{2^c}\mright))=0$ since $\sin(N\pi/2^c)=\sin(3N\pi/2^c)=0$ and $\cos(N\pi/2^c)=\cos(3N\pi/2^c)$. If $c\leq v$ and $c$ is odd, then $C_1(c,N)=2^{-1/2}(\sin\mleft({N\pi}/{2^c}\mright)+\cos\mleft({N\pi}/{2^c}\mright)-\sin\mleft({3N\pi}/{2^c}\mright)+\cos\mleft({3N\pi}/{2^c}\mright))=2^{1/2}(-1)^{\mathbf{1}_{c=v}}$ since $\sin(N\pi/2^c)=\sin(3N\pi/2^c)=0$ and $\cos(N\pi/2^c)=\cos(3N\pi/2^c)=(-1)^{\mathbf{1}_{c=v}}$. If $c=v+1$, then $C_1(c,N)$ equals
    \begin{align*}
        \frac{1}{\sqrt{2}}\left(\sin\mleft(\frac{N_2\pi}{2}\mright)+\cos\mleft(\frac{N_2\pi}{2}\mright)+(-1)^{c}\left(\sin\mleft(\frac{3N_2\pi}{2}\mright)-\cos\mleft(\frac{3N_2\pi}{2}\mright)\right)\right)=\sqrt{2}(-1)^{\frac{N_2-1}{2}}\mathbf{1}_{c\text{~odd}}
    \end{align*}
    since $\sin(N_2\pi/2)=(-1)^{(N_2-1)/2}$ and $\sin(3N_2\pi/2)=(-1)^{(N_2+1)/2}$ and $\cos(N_2\pi/2)=\cos(3N_2\pi/2)=0$. Lastly, if $c=v+2$, then
    \begin{align*}
        C_1(c,N)=\frac{1+(-1)^c}{2}\sqrt{2}\left(\sin\mleft(\frac{N_2\pi}{4}\mright)+\cos\mleft(\frac{N_2\pi}{4}\mright)\right)
    \end{align*}
    since $\sin(N_2\pi/4)=\sin(3N_2\pi/4)$ and $\cos(N_2\pi/4)=-\cos(3N_2\pi/4)$.

    Finally, suppose that $n=3$, wherein we proceed similarly to the $n=1$ case. We then have
    \begin{align*}
        \cos\mleft(\frac{n\pi}{4}-\frac{N\pi}{2^c}\mright)=\frac{1}{\sqrt{2}} \left(\sin\mleft(\frac{N\pi}{2^c}\mright)-\cos\mleft(\frac{N\pi}{2^c}\mright)\right)
    \end{align*}
    and
    \begin{align*}
        \cos\mleft(\frac{3\pi n}{4}-\frac{3\pi N}{2^c}\mright)=\frac{1}{\sqrt{2}} \left(\sin\mleft(\frac{3N\pi}{2^c}\mright)+\cos\mleft(\frac{3N\pi}{2^c}\mright)\right)
    \end{align*}
    so that
    \begin{align*}
        C_3(c,N)=\frac{1}{\sqrt{2}}\left(\sin\mleft(\frac{N\pi}{2^c}\mright)-\cos\mleft(\frac{N\pi}{2^c}\mright)+(-1)^{c}\left(\sin\mleft(\frac{3N\pi}{2^c}\mright)+\cos\mleft(\frac{3N\pi}{2^c}\mright)\right)\right).
    \end{align*}
    So if $c\leq v$ and $c$ is even, then $C_3(c,N)=2^{-1/2}(\sin(N\pi/2^c)-\cos(N\pi/2^c)+\sin(3N\pi/2^c)+\cos(3N\pi/2^c))=0$ since $\sin(N\pi/2^c)=\sin(3N\pi/2^c)=0$ and $\cos(N\pi/2^c)=\cos(3N\pi/2^c)$. If $c\leq v$ and $c$ is odd, then $C_3(c,N)=2^{-1/2}(\sin(N\pi/2^c)-\cos(N\pi/2^c)-\sin(3N\pi/2^c)-\cos(3N\pi/2^c))=-2^{1/2}(-1)^{\mathbf{1}_{c=v}}$ since $\sin(N\pi/2^c)=\sin(3N\pi/2^c)=0$ and $\cos(N\pi/2^c)=\cos(3N\pi/2^c)=(-1)^{\mathbf{1}_{c=v}}$. If $c=v+1$, then $C_3(c,N)$ equals
    \begin{align*}
        \frac{1}{\sqrt{2}}\left(\sin\mleft(\frac{N_2\pi}{2}\mright)-\cos\mleft(\frac{N_2\pi}{2}\mright)+(-1)^c\left(\sin\mleft(\frac{3N_2\pi}{2}\mright)+\cos\mleft(\frac{3N_2\pi}{2}\mright)\right)\right)=\sqrt{2}(-1)^{\frac{N_2-1}{2}}\mathbf{1}_{c\text{~odd}}
    \end{align*}
    since $\sin(N_2\pi/2)=(-1)^{(N_2-1)/2}$ and $\sin(3N_2\pi/2)=(-1)^{(N_2+1)/2}$ and $\cos(N_2\pi/2)=\cos(3N_2\pi/2)=0$. Lastly, if $c=v+2$, then
    \begin{align*}
        C_3(c,N)=\frac{1+(-1)^c}{2}\sqrt{2}\left(\sin\mleft(\frac{N_2\pi}{4}\mright)-\cos\mleft(\frac{N_2\pi}{4}\mright)\right)
    \end{align*}
    since $\sin(N_2\pi/4)=\sin(3N_2\pi/4)$ and $\cos(N_2\pi/4)=-\cos(3N_2\pi/4)$.
\end{proof}

With Lemma \ref{lemma_cosine_eval} proven, we are now ready to evaluate $\delta_2^{\square}(N,n)$. We have chosen to combine the cases $n\equiv 1\pmod{4}$ and $n\equiv3\pmod{4}$ since the values of $C_1(c,N)$ and $C_3(c,N)$ are sufficiently similar, whereas we have kept the cases $n\equiv 0\pmod{4}$ and $n\equiv 2\pmod{4}$ separate.

\begin{prop}\label{2_adic_square}
    Let $\theta=1-n/2$ and $v=v_2(N)$ and $N_2=2^{-v}N$ be the odd part of $N$. We have that $\delta_{2}^{\square}(N,n)=1+(-1)^{\lfloor n/4\rfloor}f_2(N,n)$, where,
        \begin{enumerate}[(i)]
            \item if $n\equiv0\pmod4$,
            \begin{align*}
                f_2(N,n)=\left(\frac{2^{\theta}-2^{\theta v}}{1-2^{\theta}}-2^{\theta v}\right)\mathbf{1}_{v>0},
            \end{align*}
            \item if $n\equiv2\pmod 4$,
            \begin{align*}
                f_2(N,n)=2^{\theta(v+1)}(-1)^{(N_2-1)/2},
            \end{align*}
            \item if $n$ is odd and $v$ is even, $f_2(N,n)$ equals
            \begin{align*}
                \frac{1}{2^{1/2}}\left((-1)^{\frac{n-1}{2}}\frac{2^{\theta}-2^{\theta (v+1)}}{1-2^{2\theta}}+2^{\theta(v+1)}(-1)^{\frac{N_2-1}{2}}+2^{\theta(v+2)}\left(\sin\mleft(\frac{N_2\pi}{4}\mright)+(-1)^{\frac{n-1}{2}}\cos\mleft(\frac{N_2\pi}{4}\mright)\right)\right),
            \end{align*}
            \item and if $n$ is odd and $v$ is odd,
            \begin{align*}
                f_2(N,n)=\frac{(-1)^{\frac{n-1}{2}}}{2^{1/2}}\left(\frac{2^{\theta}-2^{\theta v}}{1-2^{2\theta}}-2^{\theta v}\right).
            \end{align*}
        \end{enumerate}
\end{prop}

\begin{proof}
    Let $a\neq0$ be modulo $2^d$ and $a_2=a2^{-v_2(a)}$ be its odd part. Then, unless $v_2(a)=d-1$, 
    \begin{align}\label{quad_gauss_sum_eval_2}
        \sum_{r\bmod 2^d}e\mleft(\frac{ar^2}{2^d}\mright)=2^{v_2(a)}\sum_{t\bmod 2^{d-v_2(a)}}e\mleft(\frac{a_2 t^2}{2^{d-v_2(a)}}\mright)=2^{v_2(a)}(1+i)\varepsilon_{a_2}^{-1}\left(\frac{2^{d-v_2(a)}}{a_2}\right)2^{(d-v_2(a))/2},
    \end{align}
    where the first equality follows from the substitution $r\mapsto t+b2^{d-v_p(a)}$ with $b,t$ modulo $2^{v_p(a)}$ and the second equality follows from Gauss' formula \eqref{gauss_formula_berndt}. Of course, if $v_2(a)=d-1$, then the sum in \eqref{quad_gauss_sum_eval_2} evaluates to zero. For a function $f$, we note that we can partition the sum $\sum_{a\bmod 2^d}f(a)$ via the $2$-adic valuation of the variable $a$ by
    \begin{align}\label{partition_2_adic}
        \sum_{a\bmod 2^d}f(a)=f(0)+\sum_{c=1}^{d}\sum_{t\bmod p^{d-c}} f(2^{c-1}+t2^c).
    \end{align}
    Using this formula on the exponential sum formula for $\delta_2^{\square}(N,n)$, we have that $\delta_2^{\square}(N,n)$ is the limit as $d\to\infty$ of
    \begin{align*}
        1+2^{-nd}\sum_{c=1}^{d-1}\sum_{t\bmod 2^{d-c}}\left(\sum_{r\bmod 2^d}e\mleft(\frac{(2^{c-1}+t2^c)r^2}{2^d}\mright)\right)^n e\mleft(-\frac{(2^{c-1}+t2^c)N}{2^d}\mright)
    \end{align*}
    since the $c=d$ term is zero. Letting $1\leq c\leq d-1$, by \eqref{quad_gauss_sum_eval_2},
    \begin{align*}
        \sum_{r\bmod 2^d}e\mleft(\frac{(2^{c-1}+t2^c)r^2}{2^d}\mright)=2^{(d+c-1)/2}(1+i)\varepsilon^{-1}_{2t+1}\left(\frac{2^{d-c+1}}{2t+1}\right).
    \end{align*}
    So $\delta_2^{\square}(N,n)$ is the limit as $d\to\infty$ of
    \begin{align*}
        1+2^{-nd}\sum_{c=1}^{d-1}\sum_{t\bmod 2^{d-c}} \left(2^{(d+c-1)/2}(1+i)\varepsilon_{2t+1}^{-1}\left(\frac{2^{d-c+1}}{2t+1}\right)\right)^n e\mleft(-\frac{(2^{c-1}+t2^c)N}{2^d}\mright),
    \end{align*}
    which, after rearranging, equals
    \begin{align}\label{before_cos_equating}
        1+2^{-nd}\sum_{c=1}^{d-1}2^{n(d+c-1)/2} (1+i)^n e\mleft(\frac{-N}{2^{d-c+1}}\mright)\sum_{t\bmod 2^{d-c}} \left(\varepsilon_{2t+1}^{-1} \left(\frac{2^{d-c+1}}{2t+1}\right)\right)^n e\mleft(\frac{-tN}{2^{d-c}}\mright).
    \end{align}
    Before we can evaluate further, we must first consider the casework of the factor $\varepsilon^{-1}_{2t+1}\left(\frac{2^{d-c+1}}{2t+1}\right)$. For the Jacobi symbol, we have
    \begin{align*}
        \left(\frac{2^{d-c+1}}{2t+1}\right)=\left(\frac{2}{2t+1}\right)^{d-c+1}=\begin{cases}
            (-1)^{d-c+1} & \text{if~}t\equiv1,2~(\text{mod~} 4),\\
            1 & \text{if~}t\equiv0,3~(\text{mod~} 4).
        \end{cases}
    \end{align*}
    Thus, since
    \begin{align*}
        \varepsilon_{2t+1}^{-1}=\begin{cases}
            1 & \text{if~}t\equiv0,2~(\text{mod~} 4),\\
            -i & \text{if~}t\equiv1,3~(\text{mod~} 4),\\
        \end{cases}
    \end{align*}
    we have the evaluation
    \begin{align*}
        \varepsilon_{2t+1}^{-1}\left(\frac{2^{d-c+1}}{2t+1}\right)=\begin{cases}
            1 & \text{if~}t\equiv0~(\text{mod~} 4),\\
            i(-1)^{d-c} & \text{if~}t\equiv1~(\text{mod~} 4),\\
            (-1)^{d-c+1} & \text{if~}t\equiv2~(\text{mod~} 4),\\
            -i & \text{if~}t\equiv3~(\text{mod~} 4).
        \end{cases}
    \end{align*}
    So splitting $\sum_{t\bmod 2^{d-c}} \left(\varepsilon_{2t+1}^{-1} \left(\frac{2^{d-c+1}}{2t+1}\right)\right)^n e\mleft({-tN}/{2^{d-c}}\mright)$ into the congruence classes of $t$ modulo $4$ yields
    \begin{align*}
        &\sum_{t\bmod 2^{d-c-2}} e\mleft(\frac{-(4t)N}{2^{d-c}}\mright)+(i(-1)^{d-c})^n \sum_{t\bmod 2^{d-c-2}} e\mleft(\frac{-(4t+1)N}{2^{d-c}}\mright)\nonumber\\
        &+(-1)^{n(d-c+1)} \sum_{t\bmod 2^{d-c-2}} e\mleft(\frac{-(4t+2)N}{2^{d-c}}\mright)+(-i)^n \sum_{t\bmod 2^{d-c-2}} e\mleft(\frac{-(4t+3)N}{2^{d-c}}\mright)\nonumber\\
        &=\left(1+(i(-1)^{d-c})^n e\mleft(\frac{-N}{2^{d-c}}\mright)+(-1)^{n(d-c+1)} e\mleft(\frac{-2N}{2^{d-c}}\mright)+(-i)^n e\mleft(\frac{-3N}{2^{d-c}}\mright)\right)\sum_{t\bmod 2^{d-c-2}} e\mleft(\frac{-tN}{2^{d-c-2}}\mright)\nonumber\\
        &=\left(1+(i(-1)^{d-c})^n e\mleft(\frac{-N}{2^{d-c}}\mright)+(-1)^{n(d-c+1)} e\mleft(\frac{-2N}{2^{d-c}}\mright)+(-i)^n e\mleft(\frac{-3N}{2^{d-c}}\mright)\right)2^{d-c-2}\mathbf{1}_{v_2(N)\geq d-c-2}.
    \end{align*}
    Because of the factor $\mathbf{1}_{v_2(N)\geq d-c-2}$, we henceforth assume that $c$ is such that $v_2(N)\geq d-c-2$. Appending the factor of $2^{-n/2}(1+i)^n e(-N/2^{d-c+1})=(1/\sqrt{2}+i/\sqrt{2})^ne(-N/2^{d-c+1})$ coming from \eqref{before_cos_equating}, we now consider
    \begin{align}\label{cosine_func_1}
        \left(\frac{1}{\sqrt{2}}+\frac{i}{\sqrt{2}}\right)^n e\mleft(\frac{-N}{2^{d-c+1}}\mright)\left(1+(-i)^n e\mleft(\frac{-3N}{2^{d-c}}\mright)\right)=e\mleft(\frac{n}{8}-\frac{N}{2^{d-c+1}}\mright)+e\mleft(\frac{7n}{8}-\frac{7N}{2^{d-c+1}}\mright).
    \end{align}
    Since $v_2(N)\geq d-c-2$, we have that $8N/2^{d-c+1}\in\ZZ$. So $e(7n/8-7N/2^{d-c+1})=e(-n/8+N/2^{d-c+1})$ and thus \eqref{cosine_func_1} equals $2\cos(\pi n/4-\pi N/2^{d-c})$. Similarly, consider
    \begin{align}\label{cosine_func_2}
        \left(\frac{1}{\sqrt{2}}+\frac{i}{\sqrt{2}}\right)^n &e\mleft(\frac{-N}{2^{d-c+1}}\mright)\left((i(-1)^{d-c})^n e\mleft(\frac{-N}{2^{d-c}}\mright)+(-1)^{n(d-c+1)} e\mleft(\frac{-2N}{2^{d-c}}\mright)\right)\nonumber\\
        &=(-1)^{n(d-c)}\left(e\mleft(\frac{3n}{8}-\frac{3N}{2^{d-c+1}}\mright)+e\mleft(\frac{5n}{8}-\frac{5N}{2^{d-c+1}}\mright)\right).
    \end{align}
    Since $8N/2^{d-c+1}\in\ZZ$, we have $e({5n}/{8}-{5N}/{2^{d-c+1}})=e(-3n/8+3N/2^{d-c+1})$ and thus \eqref{cosine_func_2} equals $(-1)^{n(d-c)}2\cos(3\pi n/4-3\pi N/2^{d-c})$. Therefore, together with \eqref{before_cos_equating}, $\delta_2^{\square}(N,n)$ is the limit as $d$ tends to infinity of
    \begin{align*}
        1+2^{-nd}\sum_{c=d-v_2(N)-2}^{d-1}&2^{d-c-1+n(d+c)/2}\left(\cos\mleft(\frac{\pi n}{4}-\frac{\pi N}{2^{d-c}}\mright)+(-1)^{n(d-c)}\cos\mleft(\frac{3\pi n}{4}-\frac{3\pi N}{2^{d-c}}\mright)\right)\\
        =1+\frac{1}{2}\sum_{c=1}^{v_2(N)+2}&2^{c(1-n/2)}\left(\cos\mleft(\frac{\pi n}{4}-\frac{\pi N}{2^c}\mright)+(-1)^{nc}\cos\mleft(\frac{3\pi n}{4}-\frac{3\pi N}{2^c}\mright)\right).
    \end{align*}
    Define $C_n(c,N)=\cos\mleft({\pi n}/{4}-{\pi N}/{2^c}\mright)+(-1)^{nc}\cos\mleft({3\pi n}/{4}-{3\pi N}/{2^c}\mright)$. We proceed by invoking Lemma \ref{lemma_cosine_eval}, leveraging the relation $C_{n}(c,N)=(-1)^{\lfloor n/4\rfloor}C_{n\bmod 4}(c,N)$. We note that we will consider the $v_2(N)=0$ case separately for $n\equiv0\pmod{4}$ since we assume that $v_2(N)\in\{1,\dots,v_2(N)+2\}$. However, we do not need to consider the $v_2(N)=0$ case separately for other congruence classes of $n$ modulo $4$ since $C_n(v_2(N),N)=0$ if $v_2(N)$ is even and $n\not\equiv0\pmod{4}$. Let $\theta=1-n/2$.
    
    Let $n\equiv0\pmod{4}$. For $v_2(N)>0$,
    \begin{align*}
        \delta_2^{\square}(N,n)&=1+(-1)^{\lfloor n/4\rfloor}\frac{1}{2}\left(\sum_{c=1}^{v_2(N)-1}2^{c\theta}2-2\cdot 2^{v_2(N)\theta}\right)\\
        &=1+(-1)^{\lfloor n/4\rfloor}\left(\frac{2^{\theta}-2^{v_2(N)\theta}}{1-2^{\theta}}-2^{v_2(N)\theta}\right).
    \end{align*}
    If $v_2(N)=0$, then $C_0(1,N)=C_0(2,N)=0$, wherein we have $\delta_2^{\square}(N,n)=1$. 
    
    Let $n\equiv2\pmod{4}$. Then
    \begin{align*}
        \delta_2^{\square}(N,n)=1+(-1)^{\lfloor n/4\rfloor}\frac{1}{2}\left(2^{\theta(v_2(N)+1)}2(-1)^{(N_2-1)/2}\right)=1+(-1)^{\lfloor n/4\rfloor}\left(2^{\theta(v_2(N)+1)}(-1)^{(N_2-1)/2}\right).
    \end{align*}
    
    Finally, let $n$ be odd. Then if $v_2(N)$ is even,
    \begin{align*}
        \delta_2^{\square}(N,n)=1+(-1)^{\lfloor n/4\rfloor}\frac{1}{2}\Bigg(&\sum_{\substack{1\leq c\leq v_2(N)-1\\ c\text{~odd}}}2^{c\theta}2^{1/2}(-1)^{(n-1)/2}+2^{(v_2(N)+1)\theta}2^{1/2}(-1)^{(N_2-1)/2}\\
        &+2^{(v_2(N)+2)\theta}2^{1/2}\left(\sin\mleft(\frac{N_2\pi}{4}\mright)+(-1)^{(n-1)/2}\cos\mleft(\frac{N_2\pi}{4}\mright)\right)\Bigg)\\
        =1+(-1)^{\lfloor n/4\rfloor}\frac{1}{2^{1/2}}&\Bigg((-1)^{(n-1)/2}\frac{2^{\theta}-2^{\theta (v_2(N)+1)}}{1-2^{2\theta}}+2^{\theta(v_2(N)+1)}(-1)^{(N_2-1)/2}\\
        &+2^{(v_2(N)+2)\theta}\left(\sin\mleft(\frac{N_2\pi}{4}\mright)+(-1)^{(n-1)/2}\cos\mleft(\frac{N_2\pi}{4}\mright)\right)\Bigg),
    \end{align*}
    and if $v_2(N)$ is odd,
    \begin{align*}
        \delta_2^{\square}(N,n)&=1+(-1)^{\lfloor n/4\rfloor}\frac{1}{2}\left(\sum_{\substack{1\leq c\leq v_2(N)-2\\ c\text{~odd}}}2^{c\theta}2^{1/2}(-1)^{(n-1)/2}-2^{v_2(N)\theta}2^{1/2}(-1)^{(n-1)/2}\right)\\
        &=1+(-1)^{\lfloor n/4\rfloor}\frac{(-1)^{(n-1)/2}}{2^{1/2}} \left(\frac{2^{\theta}-2^{\theta v_2(N)}}{1-2^{2\theta}}-2^{\theta v_2(N)}\right).
    \end{align*}
\end{proof}

\subsection{The polygonal $p$-adic densities}

We now move to computing the polygonal $p$-adic densities. Let $G(A,B,C)\defeq\sum_{r\bmod C}e((Ar^2+Br)/C)$ be the generalized quadratic Gauss sum. Recall that $G(A,B,C)=0$ if $(A,C)>1$, unless $(A,C)\mid B$, wherein we have
\begin{align}\label{generalized_quad_gauss_sum}
    G(A,B,C)=(A,C)\cdot G(A/(A,C),B/(A,C),C/(A,C)).
\end{align}
This is, in fact, the only result that we will need to evaluate $\delta_p^{\ksymbol{k}}(N,n)$ when $p$ is an odd prime divisor of $k-2$. 

\begin{prop}\label{odd_adic_polygonal}
    Let $p$ be an odd prime divisor of $k-2$. Then $\delta_{p}^{\ksymbol{k}}(N,n)=p^{v_p(k-2)}$ for any $n$ and $N$.
\end{prop}

\begin{proof}
    Let $v=v_p(k-2)$. We first evaluate the exponential sum $\sum_{r=0}^{p^d-1}e\mleft({a(rp^{v}+k)^2}/{p^d}\mright)$. We have
    \begin{align*}
        \sum_{r\bmod p^d}e\mleft(\frac{a(rp^v+k)^2}{p^d}\mright)=e(ak^2/p^d)G(ap^{2v},2akp^v,p^d),
    \end{align*}
    which is zero unless $p^{\min(v_p(a)+2v,d)}=(ap^{2v},p^d)\mid 2akp^v$. This holds if and only if $v_p(2akp^v)=v_p(a)+v\geq\min(v_p(a)+2v,d)$, which holds only if $v_p(a)\geq d-v$ since $v>0$. If $v_p(a)\geq d-v$, then $ap^{2v}\equiv 2akp^v\equiv 0\pmod{p^d}$ so that $G(ap^{2v},2akp^v,p^d)=p^d$. So
    \begin{align*}
        \sum_{r\bmod p^d}e\mleft(\frac{a(rp^{v}+k)^2}{p^d}\mright)=\begin{cases}
            e({ak^2}/p^d) p^d & \text{if~}p^{d-v}\mid a,\\
            0 & \text{otherwise.}
        \end{cases}
    \end{align*}
    By utilizing this evaluation, we see that
    \begin{align*}
        \delta_p^{\ksymbol{k}}(N,n)&=\lim_{d\to\infty}p^{-nd}\sum_{a\bmod p^d}\left(\sum_{r\bmod p^d}e\mleft(\frac{a(rp^{v}+k)^2}{p^d}\mright)\right)^n e\mleft(\frac{-aX_N}{p^d}\mright)\\
        &=\lim_{d\to\infty}p^{-nd}\sum_{a\bmod p^v}\left(e\mleft(\frac{ap^{p-v}k^2}{p^d}\mright)p^d\right)^n e\mleft(\frac{-ap^{d-v}X_N}{p^d}\mright)\\
        &=\sum_{a\bmod p^v} e\mleft(\frac{ank^2}{p^v}\mright)e\mleft(\frac{-aX_N}{p^v}\mright)\\
        &=\sum_{a\bmod p^v}e\mleft(\frac{-a\left(8(k-2)N+n((k-4)^2-k^2)\right)}{p^v}\mright)\\
        &=\sum_{a\bmod p^v} e\mleft(\frac{-8a(k-2)(N-n)}{p^v}\mright)=\sum_{a\bmod p^v}1=p^v.
    \end{align*}
\end{proof}

The method of evaluating $\delta_2^{\ksymbol{k}}(N,n)$ is only slightly more involved than Proposition \ref{odd_adic_polygonal} when $k\not\equiv0\pmod{4}$, but is as involved as Proposition \ref{2_adic_square} when $k\equiv0\pmod{4}$. In fact, the evaluation of $\delta_2^{\ksymbol{k}}(N,n)$ is quite similar to $\delta_2^{\square}(N,n)$ when $k\equiv0\pmod{4}$. Before moving on to this evaluation, we first note a simple, standard result on $2$-adic generalized quadratic Gauss sums, which we will need to use in conjunction with \eqref{generalized_quad_gauss_sum} when $k$ is odd.

\begin{lemma}\label{odd_odd_GQGS}
    Let $L\in\ZZ^+$ and $A,B$ be positive odd integers. Then $G(A,B,2^L)=\mathbf{1}_{L=1}\cdot2$.
\end{lemma}

\begin{proof}
    First note for $L=1$ that $G(A,B,2^L)=1+e((A+B)/2)=2$. Now suppose that $L\geq2$. Let $g(x)=Ax^2+Bx$. Since $L\geq2$, we have $A\cdot 2^{2L-2}\equiv0\pmod{2^L}$. So since $2Ax+B$ is odd, $g(x+2^{L-1})=Ax^2+Bx+2^{L-1}(2Ax+B)+A\cdot2^{2L-2}\equiv g(x)+2^{L-1}$. So $e(g(x+2^{L-1})/2^L)=-e(g(x)/2^L)$ and hence $G(A,B,2^L)=0$.
\end{proof}

\begin{prop}\label{2_adic_polygonal}
    Let $\theta=1-n/2$ and $(X_N)_2=2^{-v_2(X_N)}X_N$ be the odd part of $X_N$. We have that $\delta_{2}^{\ksymbol{k}}(N,n)=2^{v_2(k-2)+3}+(-1)^{\lfloor n/4\rfloor}4^n T(X_N,n,k)$, where $T(X_N,n,k)=0$ if $k\not\equiv0\pmod{4}$, whereas for $k\equiv0\pmod{4}$,
        \begin{enumerate}[(i)]
            \item if $n\equiv0\pmod4$,
            \begin{align*}
                T(X_N,n,k)=\left(\frac{2^{5\theta}-2^{\theta v_2(X_N)}}{1-2^{\theta}}-2^{\theta v_2(X_N)}\right)\mathbf{1}_{v_2(X_N)>4},
            \end{align*}
            \item if $n\equiv2\pmod 4$,
            \begin{align*}
                T(X_N,n,k)=2^{\theta(v_2(X_N)+1)}(-1)^{((X_N)_2-1)/2},
            \end{align*}
            \item if $n$ is odd and $v_2(X_N)$ is even,
            \begin{align*}
                T(X_N,n,k)=\frac{1}{2^{1/2}}\Bigg((-1)^{(n-1)/2}\frac{2^{5\theta}-2^{\theta (v_2(X_N)+1)}}{1-2^{2\theta}}+2^{\theta(v_2(X_N)+1)}(-1)^{((X_N)_2-1)/2}\\
                +2^{(v_2(X_N)+2)\theta}\left(\sin\mleft(\frac{(X_N)_2\pi}{4}\mright)+(-1)^{(n-1)/2}\cos\mleft(\frac{(X_N)_2\pi}{4}\mright)\right)\Bigg),
            \end{align*}
            \item and if $n$ is odd and $v_2(X_N)$ is odd,
            \begin{align*}
                T(X_N,n,k)=\frac{(-1)^{(n-1)/2}}{2^{1/2}} \left(\frac{2^{5\theta}-2^{\theta v_2(X_N)}}{1-2^{2\theta}}-2^{\theta v_2(X_N)}\right).
            \end{align*}
        \end{enumerate}
\end{prop}

\begin{proof}
    Let $v=v_2(k-2)$. We first evaluate the exponential sum $\sum_{r\bmod 2^d} e\mleft({a(r2^{v+1}+k)^2}/{2^d}\mright)$. We have
    \begin{align*}
        \sum_{r\bmod 2^d}e\mleft(\frac{a(r2^{v+1}+k)^2}{2^d}\mright)=e(ak^2/2^d)G(a2^{2v+2},ak2^{v+2},2^d).
    \end{align*}
    
    Suppose that $k\equiv2\pmod4$. Recall that $G(a2^{2v+2},ak2^{v+2},2^d)=0$ unless $(a2^{2v+2},2^d)\mid ak2^{v+2}$. This holds if and only if $v_2(a)+v+3=v_2(a)+v_2(k)+v+2=v_2(ak2^{v+2})\geq\min(v_2(a)+2v+2,d)$, which holds only if $v_2(a)\geq d-v-3$ since $v\geq2$. For $v_2(a)\geq d-v-3$, we have $G(a2^{2v+2},ak2^{v+2},2^d)=2^d$. So
    \begin{align*}
        \sum_{r\bmod 2^d}e\mleft(\frac{a(r2^{v+1}+k)^2}{2^d}\mright)=\begin{cases}
            e(ak^2/2^d)2^d & \text{if~}2^{d-v-3}\mid a,\\
            0 & \text{otherwise.}
        \end{cases}
    \end{align*}
    By utilizing this evaluation, we see that
    \begin{align}\label{final_delta_2_kgonal_calc}
        \delta_2^{\ksymbol{k}}(N,n)&=\lim_{d\to\infty}2^{-nd}\sum_{a\bmod 2^d}\left(\sum_{r\bmod 2^d}e\mleft(\frac{a(r2^{v+1}+k)^2}{2^d}\mright)\right)^n e\mleft(\frac{-aX_N}{2^d}\mright)\nonumber\\
        &=\lim_{d\to\infty}2^{-nd}\sum_{a\bmod 2^{v+3}}\left(e\mleft(\frac{a2^{d-v-3}k^2}{2^d}\mright)2^d\right)^n e\mleft(\frac{-a2^{d-v-3}X_N}{2^d}\mright)\nonumber\\
        &=\sum_{a\bmod 2^{v+3}}e\mleft(\frac{ank^2}{2^{v+3}}\mright)e\mleft(\frac{-aX_N}{2^{v+3}}\mright)\nonumber\\
        &=\sum_{a\bmod 2^{v+3}}e\mleft(\frac{-a\left(8(k-2)N+n((k-4)^2-k^2)\right)}{2^{v+3}}\mright)\nonumber\\
        &=\sum_{a\bmod 2^{v+3}}e\mleft(\frac{-8a(k-2)(N-n)}{2^{v+3}}\mright)=\sum_{a\bmod 2^{v+3}}1=2^{v+3}.
    \end{align}
    
    Now suppose that $k$ is odd. Letting $v_2(a)<d-2$ and $a_2=a2^{-v_2(a)}$ be the odd part of $a$, by \eqref{generalized_quad_gauss_sum}, $G(4a,4ak,2^d)=2^{v_2(a)+2}G(a_2,a_2k,2^{d-v_2(a)-2})$. By Lemma \ref{odd_odd_GQGS}, $G(a_2,a_2k,2^{d-v_2(a)-2})=0$ unless $d-v_2(a)-2=1$, wherein $G(a_2,a_2k,2^{d-v_2(a)-2})=2$. So if $v_2(a)<d-2$, then $G(4a,4ak,2^d)=\mathbf{1}_{v_2(a)=d-3}\cdot 2^d$. Now letting $v_2(a)\geq d-2$, we then have $G(4a,4ak,2^d)=2^d$. So
    \begin{align*}
        \sum_{r\bmod 2^d}e\mleft(\frac{a(2r+k)^2}{2^d}\mright)=\begin{cases}
            e(ak^2/2^d)2^d & \text{if~}2^{d-3}\mid a,\\
            0 & \text{otherwise.}
        \end{cases}
    \end{align*}
    Thus, we can simply take the $v=0$ case of \eqref{final_delta_2_kgonal_calc}, yielding $\delta_2^{\ksymbol{k}}(N,n)=2^3$.

Finally, suppose that $k\equiv0\pmod{4}$. Letting $v_2(a)\leq d-4$, by \eqref{generalized_quad_gauss_sum}, $G(16a,8ak,2^d)=2^{v_2(a)+4}G(a_2,a_2k/2,2^{d-v_2(a)-4})$. Completing the square by $a_2r^2+(a_2k/2)r=a_2(r+k/4)^2-a_2k^2/16$, we have
\begin{align*}
    G(a_2,a_2k/2,2^{d-v_2(a)-4})&=e\mleft(\frac{-a_2k^2/16}{2^{d-v_2(a)-4}}\mright)\sum_{r\bmod 2^{d-v_2(a)-4}}e\mleft(\frac{a_2(r+k/4)^2}{2^{d-v_2(a)-4}}\mright)\\
    &=e\mleft(\frac{-ak^2}{2^{d}}\mright)\sum_{r\bmod 2^{d-v_2(a)-4}}e\mleft(\frac{a_2r^2}{2^{d-v_2(a)-4}}\mright).
\end{align*}
For $v_2(a)\geq d-3$, we have $G(16a,8ak,2^d)=2^d$. So if $v_2(a)\leq d-4$, then
\begin{align}\label{k_0_4_GQGS_eval}
        \sum_{r\bmod 2^d}e\mleft(\frac{a(4r+k)^2}{2^d}\mright)=2^{v_2(a)+4}\sum_{r\bmod 2^{d-v_2(a)-4}}e\mleft(\frac{a_2r^2}{2^{d-v_2(a)-4}}\mright),
\end{align}
and if $v_2(a)\geq d-3$, then $\sum_{r\bmod 2^d}e\mleft({a(4r+k)^2}/{2^d}\mright)=e\mleft({ak^2}/{2^d}\mright)2^d=2^d$. Note that when $v_2(a)=d-4$, \eqref{k_0_4_GQGS_eval} equals $2^d$ and when $v_2(a)=d-5$, \eqref{k_0_4_GQGS_eval} is zero. Using \eqref{partition_2_adic} on the exponential sum formula for $\delta_2^{\ksymbol{k}}(N,n)$, we have that $\delta_2^{\ksymbol{k}}(N,n)$ is the limit as $d\to\infty$ of
\begin{align}\label{partition_k_0_4}
    1+2^{-nd}\sum_{c=1}^{d}\sum_{t\bmod 2^{d-c}}\left(\sum_{r\bmod 2^d}e\mleft(\frac{(2^{c-1}+t2^c)(4r+k)^2}{2^d}\mright)\right)^n e\mleft(-\frac{(2^{c-1}+t2^c)X_N}{2^d}\mright).
\end{align}
Note that $v_2(X_N)\geq4$. Thus, by our evaluations, \eqref{partition_k_0_4} equals
\begin{align}\label{partition_k_0_4_eval}
    &1+2^{-nd}\sum_{c=d-3}^{d}\sum_{t\bmod 2^{d-c}}2^{nd}e\mleft(-\frac{(2^{c-1}+t2^c)X_N}{2^d}\mright)\nonumber\\
    &+2^{-nd}\sum_{c=1}^{d-5}\sum_{t\bmod 2^{d-c}}\left(2^{c+3}\sum_{r\bmod 2^{d-c-3}}e\mleft(\frac{(2t+1)r^2}{2^{d-c-3}}\mright)\right)^n e\mleft(-\frac{(2^{c-1}+t2^c)X_N}{2^d}\mright)\nonumber\\
    =&1+\sum_{c=d-3}^{d}2^{d-c}
    +2^{-nd}\sum_{c=1}^{d-5}\sum_{t\bmod 2^{d-c}}\left(2^{c+3}\sum_{r\bmod 2^{d-c-3}}e\mleft(\frac{(2t+1)r^2}{2^{d-c-3}}\mright)\right)^n e\mleft(-\frac{(2^{c-1}+t2^c)X_N}{2^d}\mright)\nonumber\\
    =&16+2^{-nd}\sum_{c=1}^{d-5}\sum_{t\bmod 2^{d-c}}\left(2^{c+3}\sum_{r\bmod 2^{d-c-3}}e\mleft(\frac{(2t+1)r^2}{2^{d-c-3}}\mright)\right)^n e\mleft(-\frac{(2^{c-1}+t2^c)X_N}{2^d}\mright).
\end{align}
Now by \eqref{gauss_formula_berndt}, $2^{c+3}\sum_{r\bmod 2^{d-c-3}}e\mleft({(2t+1)r^2}/{2^{d-c-3}}\mright)$ equals
\begin{align*}
    2^{(d+c+3)/2}(1+i)\varepsilon_{2t+1}^{-1}\left(\frac{2^{d-c-3}}{2t+1}\right)=4\cdot 2^{(d+c-1)/2}(1+i)\varepsilon_{2t+1}^{-1}\left(\frac{2^{d-c+1}}{2t+1}\right).
\end{align*}
Following our calculations in Proposition \ref{2_adic_square}, \eqref{partition_k_0_4_eval} becomes
\begin{align*}
    16+4^n 2^{-nd}\sum_{c=d-v_2(X_N)-2}^{d-5}&2^{d-c-1+n(d+c)/2}\left(\cos\mleft(\frac{\pi n}{4}-\frac{\pi X_N}{2^{d-c}}\mright)+(-1)^{n(d-c)}\cos\mleft(\frac{3\pi n}{4}-\frac{3\pi X_N}{2^{d-c}}\mright)\right)\\
        =16+\frac{4^n}{2}\sum_{c=5}^{v_2(X_N)+2}&2^{c(1-n/2)}\left(\cos\mleft(\frac{\pi n}{4}-\frac{\pi X_N}{2^c}\mright)+(-1)^{nc}\cos\mleft(\frac{3\pi n}{4}-\frac{3\pi X_N}{2^c}\mright)\right).
\end{align*}
Our final evaluations are very similar to Proposition \ref{2_adic_square}. We consider the $v_2(X_N)=4$ case separately for $n\equiv0\pmod4$ for the same reasons that we considered the $v_2(N)=0$ case separately in Proposition \ref{2_adic_square}. Let $n\equiv 0\pmod4$. For $v_2(X_N)>4$,
\begin{align*}
    \delta_2^{\ksymbol{k}}(N,n)&=16+\frac{4^n}{2}(-1)^{\lfloor n/4\rfloor}\left(\sum_{c=5}^{v_2(X_N)-1}2^{c\theta}2-2\cdot 2^{v_2(X_N)\theta}\right)\\
    &=16+(-1)^{\lfloor n/4\rfloor}4^n\left(\frac{2^{5\theta}-2^{v_2(X_N)\theta}}{1-2^{\theta}}-2^{v_2(X_N)\theta}\right).
\end{align*}
If $v_2(X_N)=4$, we have $\delta_2^{\ksymbol{k}}(N,n)=16$. 

Let $(X_N)_2=2^{-v_2(X_N)}X_N$ be the odd part of $X_N$. Suppose that $n\equiv2\pmod{4}$. Then
    \begin{align*}
        \delta_2^{\ksymbol{k}}(N,n)&=16+(-1)^{\lfloor n/4\rfloor}\frac{4^n}{2}\left(2^{\theta(v_2(X_N)+1)}2(-1)^{((X_N)_2-1)/2}\right)\\
        &=16+(-1)^{\lfloor n/4\rfloor}4^n\left(2^{\theta(v_2(X_N)+1)}(-1)^{((X_N)_2-1)/2}\right).
    \end{align*}

Finally, let $n$ be odd. Then if $v_2(X_N)$ is even,
    \begin{align*}
        \delta_2^{\ksymbol{k}}(N,n)=16+(-1)^{\lfloor n/4\rfloor}\frac{4^n}{2}\Bigg(&\sum_{\substack{5\leq c\leq v_2(X_N)-1\\ c\text{~odd}}}2^{c\theta}2^{1/2}(-1)^{(n-1)/2}+2^{(v_2(X_N)+1)\theta}2^{1/2}(-1)^{((X_N)_2-1)/2}\\
        &+2^{(v_2(X_N)+2)\theta}2^{1/2}\left(\sin\mleft(\frac{(X_N)_2\pi}{4}\mright)+(-1)^{(n-1)/2}\cos\mleft(\frac{(X_N)_2\pi}{4}\mright)\right)\Bigg)\\
        =16+(-1)^{\lfloor n/4\rfloor}4^n\frac{1}{2^{1/2}}&\Bigg((-1)^{(n-1)/2}\frac{2^{5\theta}-2^{\theta (v_2(X_N)+1)}}{1-2^{2\theta}}+2^{\theta(v_2(X_N)+1)}(-1)^{((X_N)_2-1)/2}\\
        &+2^{(v_2(X_N)+2)\theta}\left(\sin\mleft(\frac{(X_N)_2\pi}{4}\mright)+(-1)^{(n-1)/2}\cos\mleft(\frac{(X_N)_2\pi}{4}\mright)\right)\Bigg),
    \end{align*}
    and if $v_2(X_N)$ is odd,
     \begin{align*}
        \delta_2^{\ksymbol{k}}(N,n)&=16+(-1)^{\lfloor n/4\rfloor}\frac{4^n}{2}\left(\sum_{\substack{5\leq c\leq v_2(X_N)-2\\ c\text{~odd}}}2^{c\theta}2^{1/2}(-1)^{(n-1)/2}-2^{v_2(X_N)\theta}2^{1/2}(-1)^{(n-1)/2}\right)\\
        &=16+(-1)^{\lfloor n/4\rfloor}4^n\frac{(-1)^{(n-1)/2}}{2^{1/2}} \left(\frac{2^{5\theta}-2^{\theta v_2(X_N)}}{1-2^{2\theta}}-2^{\theta v_2(X_N)}\right).
    \end{align*}
\end{proof}

\section{Proof of Results}\label{section_proof_of_results}

We begin this section by proving Theorem \ref{more_general_theorem}, which follows immediately from Theorem \ref{modification_theorem} and Propositions \ref{odd_adic_square} and \ref{2_adic_square}. Indeed, with notation as in Theorem \ref{more_general_theorem}, using Remark \ref{r_equiv_remark} with these results gives
\begin{align*}
    r_n^{Q}(N)=r_{n}^{\equiv}(q_N)&=(4a)^{-n}\left(\prod_{q\mid 4a}\frac{\sigma_p^{\equiv}(q_N,n)}{\delta_p^{\square}(q_N,n)}\right) r_n^{\square}(q_N)+O(N^{(n-1)/4+\varepsilon})\\
    &=\frac{(4a)^{-n}}{\mathcal{N}_n(q_N,2a)}\left(\prod_{p\mid 4a}\delta_p^{Q}(N,n)\right)r_n^{\square}(q_N)+O(N^{(n-1)/4+\varepsilon}).
\end{align*}
Now towards Theorem \ref{main_result}, take $Q(x)=p_k(x)$. Then by Propositions \ref{odd_adic_polygonal} and \ref{2_adic_polygonal}, we have $\delta_p^{Q}(N,n)=\delta_p^{\ksymbol{k}}(N,n)=p^{v_p(k-2)}$ when $p\mid k-2$ is odd, and have $\delta_2^{Q}(N,n)=\delta_2^{\ksymbol{k}}(N,n)=2^{v_2(k-2)+3}+(-1)^{\lfloor n/4\rfloor} 4^n T(X_N,n,k)$. So $r_n^{\ksymbol{k}}(N)$ equals
\begin{align*}
    \frac{(2(k-2))^{-n}}{\mathcal{N}_n(X_N,k-2)} \left(2^{v_2(k-2)+3}+(-1)^{\lfloor \frac{n}{4}\rfloor}4^n T(X_N,n,k)\right)\left(\prod_{\substack{p\mid k-2\\ p\text{~odd}}}p^{v_p(k-2)}\right)r_n^{\square}(X_N)+O(N^{(n-1)/4+\varepsilon}).
\end{align*}
Note that
\begin{align*}
    1+(-1)^{\lfloor \frac{n}{4}\rfloor}2^{2n-v_2(k-2)-3} T(X_N,n,k)=1+(-1)^{\lfloor \frac{n}{4}\rfloor} 4^{n-2} T(X_N,n,k)
\end{align*}
for all $n,k$ since $T(X_N,n,k)=0$ if $k\not\equiv0\pmod{4}$. Thus,
\begin{align*}
        r_n^{\ksymbol{k}}(N)=\frac{1+(-1)^{\lfloor n/4\rfloor}4^{n-2}T(X_N,n,k)}{2^{n-3}(k-2)^{n-1}\mathcal{N}_n(X_N,k-2)}r_n^{\square}(X_N)+O(N^{(n-1)/4+\varepsilon})
\end{align*}
as claimed.

We now proceed to proving Corollary \ref{bringmann_corollary} since the proof of Corollary \ref{moment_estimates_corollary} will use the computations done for the proof of Corollary \ref{bringmann_corollary}. Corollary \ref{bringmann_corollary} follows immediately from Theorem \ref{main_result} and \cite[Theorem 1.1]{bringmann} up to computing $T(X_N,4,k)$ and $\mathcal{N}_4(X_N,k-2)$, which is what we now do.

\begin{proof}[Proof of Corollary \ref{bringmann_corollary}]
    Note that if $p$ is an odd prime divisor of $k-2$, then $v_p(8(k-2)N+4(k-4)^2)=0$. So by Proposition \ref{odd_adic_square}, $\delta_{p_i}^{\square}(X_N,4)=1-p_i^{-2}$. We use Propositions \ref{2_adic_square} and \ref{2_adic_polygonal} for our forthcoming evaluations.

    Suppose that $k$ is odd. Then, since $v_2(X_N)=2$,
    \begin{align*}
        1-f_2(X_N,4)=1+2^{-2}-\frac{2^{-1}-2^{-2}}{1-2^{-1}}=\frac{3}{4}.
    \end{align*}
    Noting that $T(X_N,4,k)=0$, we obtain
    \begin{align}\label{BC_k_odd}
        \frac{1-16T(X_N,4,k)}{2(k-2)^{3}\mathcal{N}_4(X_N,k-2)}=\frac{1}{2(k-2)^3}\cdot \frac{1}{3/4}\cdot \prod_{i=1}^{r}\frac{p_i^2}{p_i^2-1}=\frac{2}{3(k-2)^3}\prod_{i=1}^{r}\frac{p_i^2}{p_i^2-1}.
    \end{align}
    
    Now suppose that $k\equiv2\pmod{4}$. Then, since $v_2(X_N)=4$,
    \begin{align*}
        1-f_2(X_N,4)=1+2^{-4}-\frac{2^{-1}-2^{-4}}{1-2^{-1}}=\frac{3}{16}.
    \end{align*}
    Noting that $T(X_N,4,k)=0$, we obtain
    \begin{align}\label{BC_k_2_mod_4}
        \frac{1-16T(X_N,4,k)}{2(k-2)^{3}\mathcal{N}_4(X_N,k-2)}=\frac{1}{2(k-2)^3}\cdot \frac{1}{3/16}\cdot \prod_{i=1}^{r}\frac{p_i^2}{p_i^2-1}=\frac{8}{3(k-2)^3}\prod_{i=1}^{r}\frac{p_i^2}{p_i^2-1}.
    \end{align}

    Finally, suppose that $k\equiv0\pmod{4}$. Then
    \begin{align*}
        1-f_2(X_N,4)=1+2^{-v_2(X_N)}-\frac{2^{-1}-2^{-v_2(X_N)}}{1-2^{-1}}=3\cdot 2^{-v_2(X_N)}.
    \end{align*}
    If $N$ is even, then $v_2(X_N)\geq5$, whence we have
    \begin{align*}
        1-16T(X_N,4,k)=1-16\left(\frac{2^{-5}-2^{-v_2(X_N)}}{1-2^{-1}}-2^{-v_2(X_N)}\right)=3\cdot 2^{4-v_2(X_N)}.
    \end{align*}
    On the other hand, if $N$ is odd, then $v_2(X_N)=4$, whence we have $1-16T(X_N,4,k)=1$. Thus, if $N$ is even,
    \begin{align}\label{BC_k_4_mod_4_N_even}
        \frac{1-16T(X_N,4,k)}{2(k-2)^{3}\mathcal{N}_4(X_N,k-2)}=\frac{1}{2(k-2)^3}\cdot\frac{3\cdot 2^{4-v_2(X_N)}}{3\cdot 2^{-v_2(X_N)}}\cdot\prod_{i=1}^{r}\frac{p_i^2}{p_i^2-1}=\frac{8}{(k-2)^3}\prod_{i=1}^{r}\frac{p_i^2}{p_i^2-1},
    \end{align}
    and if $N$ is odd,
    \begin{align}\label{BC_k_4_mod_4_N_odd}
        \frac{1-16T(X_N,4,k)}{2(k-2)^{3}\mathcal{N}_4(X_N,k-2)}=\frac{1}{2(k-2)^3}\cdot\frac{1}{3\cdot 2^{-4}}\cdot\prod_{i=1}^{r}\frac{p_i^2}{p_i^2-1}=\frac{8}{3(k-2)^3}\prod_{i=1}^{r}\frac{p_i^2}{p_i^2-1}.
    \end{align}
    Theorem \ref{main_result} and \cite[Theorem 1.1]{bringmann} yield
    \begin{align*}
        r_{4,+}^{\ksymbol{k}}(N)=\frac{1}{16}\cdot\frac{1-16T(X_N,4,k)}{2(k-2)^{3}\mathcal{N}_4(X_N,k-2)} r_4^{\square}(X_N)+O(N^{15/16+\varepsilon}),
    \end{align*}
    thereby giving Corollary \ref{bringmann_corollary} via \eqref{BC_k_odd}--\eqref{BC_k_4_mod_4_N_odd}.
\end{proof}

We now move on to Corollary \ref{moment_estimates_corollary}, using Theorem \ref{main_result} to reduce the problem to estimating $\sum_{N\leq x}r_n^{\square}(X_N)$ and $\sum_{N\leq x}(-1)^N r_n^{\square}(X_N)$. From the computations done in the proof of Corollary \ref{bringmann_corollary}, we will need to split the proof of Corollary \ref{moment_estimates_corollary} into cases: when $k$ is odd, when $k\equiv2\pmod{4}$, and when $k\equiv0\pmod{4}$. This makes the proof somewhat lengthy, but the cases are similar, and the arguments are standard. In particular, since we appeal to a Perron's formula argument, before moving forward in proving Corollary \ref{moment_estimates_corollary}, it is prudent to calculate first the Dirichlet series of $r_4^{\square}(4N)^2\chi(N)$ and $(-1)^Nr_4^{\square}(4N)^2\chi(N)$ for $\chi$ a Dirichlet character. The argument of $r_4^{\square}$ is scaled since the common difference and constant term of the arithmetic progression $X_N$ are not coprime. We refer the reader interested in estimating other moments of $r_n^{\ksymbol{k}}$ (especially the second moments for $n=6$ and $n=8$) to \cite{dirichlet_sums_of_squares}, which we have found helpful in generating intuition for the calculation of $\mathcal{R}_4(s,\chi)$ and $\mathcal{A}_4(s,\chi)$.

\begin{lemma}\label{dirichlet_r_4_4N_chi_N}
    Let $\chi$ be a Dirichlet character and $\mathcal{R}_4(s,\chi)=\sum_{N\geq1}{r_4^{\square}(4N)^2\chi(N)}/{N^s}$ and $\mathcal{A}_4(s,\chi)=\sum_{N\geq1}{(-1)^Nr_4^{\square}(4N)^2\chi(N)}/{N^s}$. Then for $\mathfrak{R}(s)>3$,
    \begin{align*}
        \mathcal{R}_4(s,\chi)=576\frac{(1-\chi(2)2^{1-s})^2(1-\chi(2)2^{2-s})}{1-\chi(2)^2 2^{2-2s}}\cdot\frac{L(s,\chi)L(s-1,\chi)^2 L(s-2,\chi)}{L(2s-2,\chi^2)}
    \end{align*}
    and $\mathcal{A}_4(s,\chi)=(\chi(2)2^{1-s}-1)\mathcal{R}_4(s,\chi)$.
\end{lemma}

\begin{proof}
    We note that $\sum_{N\geq1}{r_4^{\square}(4N)^2\chi(N)}/{N^s}$ is absolutely convergent for $\mathfrak{R}(s)>3$ since $r_4^{\square}(N)\ll N\log\log N$. Letting $N_2=2^{-v_2(N)}N$ be the odd part of $N$, recall the identity $r_4^{\square}(4N)=24\sigma(N_2)$. For each prime $p$, let $z_p=\chi(p) p^{-s}$. Then
    \begin{align*}
        &\sum_{N\geq1}\frac{\sigma(N_2)^2\chi(N)}{N^s}=\left(\sum_{v\geq0}\sigma(1)^2 z_2^v\right)\prod_{p\text{~odd}}\sum_{v\geq0} \sigma(p^v)^2 z_p^v\\
        =~&\frac{1}{1-z_2}\prod_{p\text{~odd}}\sum_{v\geq0} \left(\frac{1-p^{v+1}}{1-p}\right)^2 z_p^v=\frac{1}{1-z_2}\prod_{p\text{~odd}}\frac{1-p^2 z_p^2}{(1-z_p)(1-pz_p)^2(1-p^2z_p)}.
    \end{align*}
    So
    \begin{align*}
        \mathcal{R}_4(s,\chi)&=576\frac{1}{1-\chi(2)2^{-s}}\prod_{p\text{~odd}}\frac{1-\chi(p)^2 p^{2-2s}}{(1-\chi(p) p^{-s})(1-\chi(p) p^{1-s})^2(1-\chi(p)p^{2-s})}\\
        &=576\frac{(1-\chi(2)2^{1-s})^2(1-\chi(2)2^{2-s})}{1-\chi(2)^2 2^{2-2s}}\cdot\frac{L(s,\chi)L(s-1,\chi)^2 L(s-2,\chi)}{L(2s-2,\chi^2)}.
    \end{align*}
    Now for $\mathcal{A}_4(s,\chi)$, we have
    \begin{align*}
        \mathcal{A}_4(s,\chi)&=576\sum_{N\geq1}\frac{(-1)^N\sigma(N_2)^2\chi(N)}{N^s}=576\left(-\sigma(1)^2+\sum_{v\geq1}\sigma(1)^2 z_2^v\right)\prod_{p\text{~odd}}\sum_{v\geq0} \sigma(p^v)^2 z_p^v\\
        &=\frac{2z_2-1}{1-z_2}\left(\frac{1}{1-z_2}\right)^{-1} \mathcal{R}_4(s,\chi)=(\chi(2)2^{1-s}-1)\mathcal{R}_4(s,\chi),
    \end{align*}
    as desired.
\end{proof}

\begin{proof}[Proof of Corollary \ref{moment_estimates_corollary}]
    Let $p_1,\dots,p_r$ be the odd prime divisors of $k-2$. We first handle the case where $k$ is odd. By Theorem \ref{main_result} and \eqref{BC_k_odd},
    \begin{align*}
        \sum_{N\leq x}r_4^{\ksymbol{k}}(N)^2&=\sum_{N\leq x}\left(\frac{2}{3(k-2)^3}\left(\prod_{i=1}^{r}\frac{p_i^2}{p_i^2-1}\right)r_4^{\square}(8(k-2)N+4(k-4)^2)+O(N^{3/4+\varepsilon})\right)^2\\
        &=\frac{4}{9(k-2)^6}\left(\prod_{i=1}^{r}\frac{p_i^2}{p_i^2-1}\right)^2\sum_{N\leq x}r_4^{\square}(8(k-2)N+4(k-4)^2)^2+O(x^{11/4+\varepsilon}),
    \end{align*}
    where we have the error of $O(x^{11/4+\varepsilon})$ since $\sum_{N\leq x}r_4^{\square}(N)N^{3/4+\varepsilon}\ll x^{11/4+\varepsilon}$, which follows by summation by parts since $\sum_{N\leq x}r_4^{\square}(N)\ll x^2$. Since $(4(k-4)^2,8(k-2))=4$, we have the orthogonality relation
    \begin{align*}
        \mathbf{1}_{N\equiv4(k-4)^2\bmod{8(k-2)}}=\frac{\mathbf{1}_{4\mid N}}{\phi(2(k-2))}\sum_{\chi} \overline{\chi((k-4)^2)}\chi(N/4),
    \end{align*}
    where the sum is taken over the Dirichlet characters modulo $2(k-2)$. Letting $y=2(k-2)x+(k-4)^2$, we then have for some constant $B_k$ that $\sum_{N\leq x}r_4^{\square}(8(k-2)N+4(k-4)^2)^2+B_k$ equals
    \begin{align*}
        \sum_{N\leq 4y}r_4^{\square}(N)^2\mathbf{1}_{N\equiv4(k-4)^2\bmod{8(k-2)}}=\frac{1}{\phi(2(k-2))}\sum_{\chi} \overline{\chi((k-4)^2)}\sum_{N\leq y}r_4^{\square}(4N)^2\chi(N).
    \end{align*}
    We proceed by estimating $\sum_{N\leq y}r_4^{\square}(4N)^2\chi(N)$. By Lemma \ref{dirichlet_r_4_4N_chi_N},
    \begin{align*}
        \mathcal{R}_4(s,\chi)\defeq 576\frac{(1-\chi(2)2^{1-s})^2(1-\chi(2)2^{2-s})}{1-\chi(2)^2 2^{2-2s}}\cdot\frac{L(s,\chi)L(s-1,\chi)^2 L(s-2,\chi)}{L(2s-2,\chi^2)}
    \end{align*}
    is the meromorphic extension of $\sum_{N\geq1}{r_4^{\square}(4N)^2\chi(N)}/{N^s}$ to $\CC$. By Perron's formula (see \cite[Corollary 5.3]{Montgomery_Vaughan_2006}), we have
    \begin{align*}
        \sum_{N\leq y}r_4^{\square}(4N)^2\chi(N)=\frac{1}{2\pi i}\int_{\sigma_0-iT}^{\sigma_0+iT}\mathcal{R}_4(s,\chi)\frac{y^s}{s}ds+O(y^{3+\varepsilon}/T)
    \end{align*}
    for $\sigma_0>3$. Let $G(s)=\mathcal{R}_4(s,\chi){y^s}/{s}$ and let $\delta,\delta'>0$ be arbitrary. Taking $\sigma_0=2+\delta$, we write by the residue theorem the vertical integral $\int_{\sigma_0-iT}^{\sigma_0+iT}G(s)ds$ as
    \begin{align*}
        2\pi i\text{Res}(G(s),3)-\int_{(3+\delta)+iT}^{(2+\delta')+iT}G(s)ds-\int_{(2+\delta')+iT}^{(2+\delta')-iT}G(s)ds-\int_{(2+\delta')-iT}^{(3+\delta)-iT}G(s)ds.
    \end{align*}
    We proceed by bounding the three integrals above, which we henceforth denote as $E_1$, $E_2$, and $E_3$, respectively. We note that for $\Re(s)\geq 2+\delta'$, we have
    \begin{align*}
        \frac{(1-\chi(2)2^{1-s})^2(1-\chi(2)2^{2-s})}{1-\chi(2)^2 2^{2-2s}},L(s,\chi),L(s-1,\chi),L(2s-2,\chi)^{-1}\ll1.
    \end{align*}
    Thus, it suffices to estimate $L(s-2,\chi)$ along the contours of $E_1$, $E_2$, and $E_3$. For this, we use the standard convexity bound
    \begin{align}\label{convexity_bound}
        L(\sigma+it,\chi)\ll (1+|t|)^{\max(0,(1-\sigma)/2)+\varepsilon}
    \end{align}
    for $\sigma\geq0$. The estimation of the horizontal integrals $E_1$ and $E_3$ is identical, so we only consider $E_3$. By \eqref{convexity_bound}, we have
    \begin{align*}
        E_3&=\int_{2+\delta'}^{3+\delta}G(u-iT)du=\int_{2+\delta'}^{3}G(u-iT)du+\int_{3}^{3+\delta}G(u-iT)du\\
        &\ll\int_{2+\delta'}^{3}T^{(1-(u-2))/2+\varepsilon}\frac{y^u}{T}du+\int_{3}^{3+\delta} T^{\varepsilon}\frac{y^u}{T}du\ll y^{3+\delta}/T^{1-\varepsilon}.
    \end{align*} 
    Following in identical fashion, we have $E_1\ll y^{3+\delta}/T^{1-\varepsilon}$. Now for the vertical integral $E_2$, we have by \eqref{convexity_bound} that
    \begin{align*}
        E_2&=\int_{-T}^{T}G(2+\delta'-it)dt\ll\int_{-T}^{T}(1+|t|)^{(1-\delta')/2+\varepsilon}\frac{y^{2+\delta'}}{1+|t|}dt\\
        &=y^{2+\delta'}\int_{-T}^{T}(1+|t|)^{-(1+\delta')/2+\varepsilon}dt\ll y^{2+\delta'} T^{1/2+\varepsilon}.
    \end{align*}
    Since $\delta,\delta'>0$ were arbitrary, by taking $T=y^{2/3}$ we obtain
    \begin{align*}
        \sum_{N\leq y}r_4^{\square}(4N)^2\chi(N)=\text{Res}(\mathcal{R}_4(s,\chi)y^s/s,3)+O(y^{7/3+\varepsilon}).
    \end{align*}
    The error estimate of $O(y^{7/3+\varepsilon})$ will suffice for our purposes. Let $\chi_0$ be the principal character modulo $2(k-2)$ and
    \begin{align*}
        \Pi\defeq\prod_{i=1}^{r}\frac{(1-p_i^{-3})(1-p_i^{-2})^2(1-p_i^{-1})}{(1-p_i^{-4})}.
    \end{align*}
    Upon noting that $\chi_0(2)=0$, we have the evaluation
    \begin{align*}
        \text{Res}(\mathcal{R}_4(s,\chi_0)y^s/s,3)=576\frac{\zeta(3)\zeta(2)^2}{\zeta(4)}\prod_{p\mid 2(k-2)}\frac{(1-p^{-3})(1-p^{-2})^2(1-p^{-1})}{(1-p^{-4})}\frac{y^3}{3}
        =126\zeta(3)\Pi y^3.
    \end{align*}
    Therefore
    \begin{align*}
        \frac{1}{\phi(2(k-2))}\sum_{\chi} \overline{\chi((k-4)^2)}\sum_{N\leq y}r_4^{\square}(4N)^2\chi(N)= \frac{126\zeta(3)\Pi}{\phi(2(k-2))}y^3+O(y^{7/3+\varepsilon}).
    \end{align*}
    Thus,
    \begin{align*}
        \sum_{N\leq x} r_4^{\ksymbol{k}}(N)^2&=\frac{4}{9(k-2)^6}\left(\prod_{i=1}^{r}\frac{p_i^2}{p_i^2-1}\right)^2 \frac{126\zeta(3)\Pi}{\phi(2(k-2))}(2(k-2)x)^3+O(x^{11/4+\varepsilon})\\
        &=\frac{56\zeta(3)}{(k-2)^7}\prod_{i=1}^{r}\frac{1-p_i^{-3}}{1-p_i^{-4}}(2(k-2)x)^3+O(x^{11/4+\varepsilon})\\
        &=\frac{448\zeta(3)}{(k-2)^4}\prod_{i=1}^{r}\frac{1-p_i^{-3}}{1-p_i^{-4}} x^3+O(x^{11/4+\varepsilon}).
    \end{align*}

    We now handle the case where $k\equiv2\pmod{4}$, for which we may borrow much of the computation from the $k$ odd case. By Theorem \ref{main_result} and \eqref{BC_k_2_mod_4},
    \begin{align*}
        \sum_{N\leq x}r_4^{\ksymbol{k}}(N)^2&=\sum_{N\leq x}\left(\frac{8}{3(k-2)^3}\left(\prod_{i=1}^{r}\frac{p_i^2}{p_i^2-1}\right)r_4^{\square}(8(k-2)N+4(k-4)^2)+O(N^{3/4+\varepsilon})\right)^2\\
        &=\frac{64}{9(k-2)^6}\left(\prod_{i=1}^{r}\frac{p_i^2}{p_i^2-1}\right)^2\sum_{N\leq x}r_4^{\square}(8(k-2)N+4(k-4)^2)^2+O(x^{11/4+\varepsilon}).
    \end{align*}
    Since $(4(k-4)^2,8(k-2))=16$, we have the orthogonality relation
    \begin{align*}
        \mathbf{1}_{N\equiv4(k-4)^2\bmod{8(k-2)}}=\frac{\mathbf{1}_{16\mid N}}{\phi((k-2)/2)}\sum_{\chi} \overline{\chi((k-4)^2/4)}\chi(N/16),
    \end{align*}
    where the sum is taken over the Dirichlet characters modulo $(k-2)/2$. Letting $y=((k-2)/2)x+(k-4)^2/4$, we then have for some constant $B_k$ that $\sum_{N\leq x}r_4^{\square}(8(k-2)N+4(k-4)^2)^2+B_k$ equals
    \begin{align*}
        \sum_{N\leq 16y}r_4^{\square}(N)^2\mathbf{1}_{N\equiv4(k-4)^2\bmod{8(k-2)}}=\frac{1}{\phi((k-2)/2)}\sum_{\chi} \overline{\chi((k-4)^2/4)}\sum_{N\leq y}r_4^{\square}(16N)^2\chi(N).
    \end{align*}
    Since $r_4^{\square}(16N)=r_4^{\square}(4N)$, the estimation is identical to that in the $k$ odd case. Since ${2\mid(k-2)/2}$, we still have $\chi_0(2)=0$, where $\chi_0$ is the principal character modulo $(k-2)/2$. So
    \begin{align*}
        \sum_{N\leq y}r_4^{\square}(16N)^2\chi(N)= \mathbf{1}_{\chi=\chi_0}126\zeta(3)\Pi y^3+O(y^{7/3+\varepsilon})
    \end{align*} 
    and hence
    \begin{align*}
        \frac{1}{\phi((k-2)/2)}\sum_{\chi} \overline{\chi((k-4)^2/4)}\sum_{N\leq y}r_4^{\square}(16N)^2\chi(N)=\frac{126\zeta(3)\Pi}{\phi((k-2)/2)}y^3+O(y^{7/3+\varepsilon}).
    \end{align*}
    Thus,
    \begin{align*}
        \sum_{N\leq x} r_4^{\ksymbol{k}}(N)^2&=\frac{64}{9(k-2)^6}\left(\prod_{i=1}^{r}\frac{p_i^2}{p_i^2-1}\right)^2 \frac{126\zeta(3)\Pi}{\phi\mleft(\frac{k-2}{2}\mright)} \left(\frac{(k-2)x}{2}\right)^3+O(x^{11/4+\varepsilon})\\
        &=\frac{3584\zeta(3)}{(k-2)^7}\prod_{i=1}^{r}\frac{1-p_i^{-3}}{1-p_{i}^{-4}}\left(\frac{(k-2)x}{2}\right)^3+O(x^{11/4+\varepsilon})\\
        &=\frac{448\zeta(3)}{(k-2)^4}\prod_{i=1}^{r}\frac{1-p_i^{-3}}{1-p_{i}^{-4}} x^3+O(x^{11/4+\varepsilon}).
    \end{align*}
    
    We finally handle the case where $k\equiv0\pmod{4}$, which is more involved. By Theorem \ref{main_result}, \eqref{BC_k_4_mod_4_N_even}, and \eqref{BC_k_4_mod_4_N_odd},
    \begin{align*}
        \sum_{N\leq x}r_4^{\ksymbol{k}}(N)^2&=\sum_{N\leq x}\left(\frac{8(2+(-1)^N)}{3(k-2)^3}\left(\prod_{i=1}^{r}\frac{p_i^2}{p_i^2-1}\right)r_4^{\square}(8(k-2)N+4(k-4)^2)+O(N^{3/4+\varepsilon})\right)^2\\
        &=\sum_{N\leq x}\frac{320}{9(k-2)^6}\left(\prod_{i=1}^{r}\frac{p_i^2}{p_i^2-1}\right)^2r_4^{\square}(8(k-2)N+4(k-4)^2)^2\\
        &+\sum_{N\leq x}\frac{256}{9(k-2)^6}\left(\prod_{i=1}^{r}\frac{p_i^2}{p_i^2-1}\right)^2 (-1)^N r_4^{\square}(8(k-2)N+4(k-4)^2)^2+O(x^{11/4+\varepsilon}).
    \end{align*}
    We first handle $\sum_{N\leq x}r_4^{\square}(8(k-2)N+4(k-4)^2)^2$. Let $y=((k-2)/2)x+(k-4)^2/4$. Since $(4(k-4)^2,8(k-2))=16$, we have for some constant $B_k$ that $\sum_{N\leq x}r_4^{\square}(8(k-2)N+4(k-4)^2)^2+B_k$ equals
    \begin{align*}
        \sum_{N\leq 16y}r_4^{\square}(N)^2\mathbf{1}_{N\equiv4(k-4)^2\bmod{8(k-2)}}=\frac{1}{\phi((k-2)/2)}\sum_{\chi} \overline{\chi((k-4)^2/4)}\sum_{N\leq y}r_4^{\square}(16N)^2\chi(N),
    \end{align*}
    where the sum over $\chi$ is over the Dirichlet characters modulo $(k-2)/2$. Since $r_4^{\square}(16N)=r_4^{\square}(4N)$, the estimation is nearly identical to that in the $k$ odd case. However, since $(k-2)/2$ is odd, we now have $\chi_0(2)=1$, where $\chi_0$ is the principal character modulo $(k-2)/2$. So
    \begin{align*}
        \sum_{N\leq y}r_4^{\square}(16N)^2\chi(N)&=
        \text{Res}\mleft(\mathcal{R}_4(s,\chi_0)\frac{y^s}{s},3\mright)+O(y^{7/3+\varepsilon})\\
        &=\mathbf{1}_{\chi=\chi_0}576\frac{(1-2^{-2})^2(1-2^{-1})}{1-2^{-4}}\frac{\zeta(3)\zeta(2)^2}{\zeta(4)}\Pi\frac{y^3}{3}+O(y^{7/3+\varepsilon})\\
        &=\mathbf{1}_{\chi=\chi_0}144\zeta(3)\Pi y^3+O(y^{7/3+\varepsilon}).
    \end{align*}
    Hence
    \begin{align*}
        \frac{1}{\phi(\frac{k-2}{2})}\sum_{\chi} \overline{\chi\mleft(\frac{(k-4)^2}{4}\mright)}\sum_{N\leq y}r_4^{\square}(16N)^2\chi(N)=\frac{144\zeta(3)\Pi}{\phi(\frac{k-2}{2})}y^3+O(y^{7/3+\varepsilon}).
    \end{align*}
    We now handle $\sum_{N\leq x}(-1)^N r_4^{\square}(8(k-2)N+4(k-4)^2)^2$. For some constant $B_k'$, we have
    \begin{align*}
        \sum_{N\leq x}(-1)^N r_4^{\square}(8(k-2)N+4(k-4)^2)^2+B_k'=\sum_{N\leq 16y}(-1)^{\frac{N-4(k-4)^2}{8(k-2)}}r_4^{\square}(N)^2\mathbf{1}_{N\equiv4(k-4)^2\bmod{8(k-2)}}.
    \end{align*}
    Now for $N\equiv 4(k-4)^2\pmod{8(k-2)}$,
    \begin{align*}
        \frac{N-4(k-4)^2}{8(k-2)}=\frac{(N/16)-((k-4)^2/4)}{(k-2)/2}\equiv(N/16)-((k-4)^2/4)\equiv N/16\pmod{2}.
    \end{align*}
    So $\sum_{N\leq x}(-1)^N r_4^{\square}(8(k-2)N+4(k-4)^2)^2+B_k'$ equals
    \begin{align*}
        \frac{1}{\phi((k-2)/2)}\sum_{\chi} \overline{\chi((k-4)^2/4)}\sum_{N\leq y}(-1)^N r_4^{\square}(16N)^2\chi(N).
    \end{align*}
    By Lemma \ref{dirichlet_r_4_4N_chi_N}, we have that $\mathcal{A}_4(s,\chi)\defeq(\chi(2)2^{1-s}-1)\mathcal{R}_4(s,\chi)$ is the meromorphic extension of $\sum_{N\geq1}(-1)^N r_4^{\square}(4N)\chi(N)/N^s=\sum_{N\geq1}(-1)^N r_4^{\square}(16N)\chi(N)/N^s$ to $\CC$. Since $\chi(2)2^{1-s}-1$ is uniformly bounded on $\Re(s)\geq 2+\delta'$ for any $\delta'>0$, the same Perron's formula argument from earlier yields
    \begin{align*}
        \sum_{N\leq y}(-1)^N r_4^{\square}(16N)^2\chi(N)&=\text{Res}(\mathcal{A}_4(s,\chi)y^s/s,3)+O(y^{7/3+\varepsilon})\\
        &=\mathbf{1}_{\chi=\chi_0}(-108\zeta(3)\Pi) y^3+O(y^{7/3+\varepsilon}).
    \end{align*}
    Hence,
    \begin{align*}
        \frac{1}{\phi(\frac{k-2}{2})}\sum_{\chi} \overline{\chi\mleft(\frac{(k-4)^2}{4}\mright)}\sum_{N\leq y}(-1)^N r_4^{\square}(16N)^2\chi(N)= -\frac{108\zeta(3)\Pi}{\phi(\frac{k-2}{2})}y^3+O(y^{7/3+\varepsilon}).
    \end{align*}
    Thus,
    \begin{align*}
        \sum_{N\leq x} r_4^{\ksymbol{k}}(N)^2&= \frac{320}{9(k-2)^6}\left(\prod_{i=1}^{r}\frac{p_i^2}{p_i^2-1}\right)^2 \frac{144\zeta(3)\Pi}{\phi(\frac{k-2}{2})} \left(\frac{(k-2)x}{2}\right)^3\\
        &-\frac{256}{9(k-2)^6}\left(\prod_{i=1}^{r}\frac{p_i^2}{p_i^2-1}\right)^2 \frac{108\zeta(3)\Pi}{\phi(\frac{k-2}{2})}\left(\frac{(k-2)x}{2}\right)^3+O(x^{11/4+\varepsilon})\\
        &=\left(320\cdot 288-256\cdot 216\right)\frac{\zeta(3)}{9(k-2)^7}\prod_{i=1}^{r}\frac{1-p_i^{-3}}{1-p_i^{-4}}\left(\frac{(k-2)x}{2}\right)^3+O(x^{11/4+\varepsilon})\\
        &=\frac{512\zeta(3)}{(k-2)^4}\prod_{i=1}^{r}\frac{1-p_i^{-3}}{1-p_i^{-4}} x^3+O(x^{11/4+\varepsilon}).
    \end{align*}
\end{proof}

Finally, we turn to Corollary \ref{conv_2_adically_corollary}. Before giving its proof, we first note a simple lemma, which follows from Jacobi's four-square theorem \cite{jacobi_exact_formulas}.

\begin{lemma}\label{conv_2_adically_lemma_squares}
    If a strictly increasing infinite subsequence $a_N$ of positive integers does not converge $2$-adically to $0$, then $r_4^{\square}(a_N)\geq \varepsilon\cdot N$ infinitely often for some $\varepsilon>0$.
\end{lemma}

\begin{proof}
    Suppose that $a_N$ is a strictly increasing infinite subsequence that does not converge $2$-adically to $0$. Then there exists an $\varepsilon>0$ such that $2^{-v_2(a_{N})}\geq\varepsilon$ for infinitely many $N$. So
    \begin{align*}
        r_4^{\square}(a_{N})=8\sum_{4\nmid d\mid a_{N}}d\geq 8\sigma(2^{-v_2(a_{N})}a_{N}) > 2^{-v_2(a_{N})}a_{N}\geq \varepsilon\cdot a_{N}\geq \varepsilon\cdot N
    \end{align*}
    infinitely often.
\end{proof}

We now prove Corollary \ref{conv_2_adically_corollary} by using Corollary \ref{bringmann_corollary} and Lemma \ref{conv_2_adically_lemma_squares}. 

\begin{proof}[Proof of Corollary \ref{conv_2_adically_corollary}]
    Suppose that $k\equiv0\pmod{4}$. Suppose that $a_N$ is a strictly increasing infinite subsequence of positive integers that does not converge $2$-adically to $(k-4)^2/(4-2k)$. We note that $(k-4)^2/(4-2k)$ is indeed a $2$-adic integer, which is easy to see by writing $(k-4)^2/(4-2k)=4(r-1)^2/(1-2r)$, where $r=k/4$. Since $a_N$ does not converge $2$-adically to $(k-4)^2/(4-2k)$, we have that $8(k-2)a_N+4(k-4)^2$ does not converge $2$-adically to $0$. Thus, by Lemma \ref{conv_2_adically_lemma_squares}, $r_4^{\square}(8(k-2)a_N+4(k-4)^2)\geq \varepsilon\cdot N$ infinitely often for some $\varepsilon>0$. So by Corollary \ref{bringmann_corollary}, $r_{4,+}^{\ksymbol{k}}(a_N)$ is unbounded.

    Suppose that $k\not\equiv0\pmod{4}$. Let $b_N$ be an arbitrary strictly increasing infinite subsequence of positive integers. If $k$ is odd, then $v_2(8(k-2)b_N+4(k-4)^2)=2$, and if $k\equiv2\pmod{4}$, then $v_2(8(k-2)b_N+4(k-4)^2)=4$. In either case, $8(k-2)b_N+4(k-4)^2$ does not converge $2$-adically to $0$. So by Lemma \ref{conv_2_adically_lemma_squares} and Corollary \ref{bringmann_corollary}, $r_{4,+}^{\ksymbol{k}}(b_N)$ is unbounded.
\end{proof}

\section{Directions for Future Work}\label{directions_for_future_work}

A natural continuation of our work is to attempt to improve the $O(N^{(n-1)/4+\varepsilon})$ error in Theorem \ref{main_result}. Since we are restricting to the class of quadratics $p_k$, it is reasonable for one to hope for improvements to the error, at least in some cases. Curiously, we have verified this to be the case in the extreme when $n=4$ and $k=6$, in that the error is in fact zero. Indeed, as stated earlier, we have the exact identity
\begin{align}\label{exact_identity_r_4_6}
    r_4^{\ksymbol{6}}(N)=\frac{1}{24}r_4^{\square}(32N+16).
\end{align}
This is not too difficult to prove elementarily. We show that the solutions to $\sum_{i=1}^{4}p_6(x_i)=N$ over $\ZZ^4$ are equinumerous to the solutions to $\sum_{i=1}^{4}p_3(m_i)=N$ over $\NN_0^4$. Completing the square to $\sum_{i=1}^{4}p_6(x_i)=N$ gives the equation $\sum_{i=1}^{4}x_i^2=32N+16$ with the congruence restriction $x_i\equiv6\pmod{8}$. Since $x_i\equiv6\pmod{8}$ is equivalent to $x_i/2=2y_i+1$ for some odd integer $y_i$, the equation becomes $\sum_{i=1}^{4}(2y_i+1)^2=8N+4$, where each $y_i$ is odd. Since $(2y_i+1)^2=8p_3(y_i)+1$, we obtain the equation $\sum_{i=1}^{4}p_3(y_i)=N$ where each $y_i$ is odd. Since $p_3(y_i)=p_3(-y_i-1)$, we then obtain the equation $\sum_{i=1}^{4}p_3(m_i)=N$, where each $m_i\in\NN_0$. Finally, we obtain \eqref{exact_identity_r_4_6} by noting that $r_{4,+}^{\ksymbol{3}}(N)=\sigma(2N+1)=(24)^{-1}r_4^{\square}(32N+16)$, where the first equality is a classical result of Legendre \cite{triangular_legendre} and the second equality follows from the identity $r_4^{\square}(2N)=24\sigma(N_2)$, where $N_2=2^{-v_2(N)}N$ is the odd part of $N$.

The case where $n=4$ and $k=6$ is obviously an extreme example in error improvement, but it does prompt the following question: for which values of $n$ and $k$ do we have that $r_n^{\ksymbol{k}}(N)$ is exactly equal to the main term in Theorem \ref{main_result}?

A more straightforward yet still interesting direction for future work is to estimate more moments of $r_n^{\ksymbol{k}}$ using Theorem \ref{main_result}. In addition to the proof of Corollary \ref{moment_estimates_corollary}, we refer the reader interested in doing so to \cite{dirichlet_sums_of_squares}. We also suggest trying to improve the error of \eqref{second_moment_of_r_4_k}, though this would likely have to be done with other methods that circumvent the use of Theorem \ref{main_result}.

For what we anticipate to be a more substantial undertaking, we also suggest generalizing Corollary \ref{bringmann_corollary} to all $n\geq4$. This generalization of Corollary \ref{bringmann_corollary} is essentially equivalent to proving the heuristic that $r_{n,+}^{\ksymbol{k}}(N)=2^{-n} r_{n}^{\ksymbol{k}}(N)+o(N^{n/2-1})$. Hence, we expect this direction to be quite challenging.

Lastly, we suggest furthering the $k\equiv0\pmod{4}$ case of Corollary \ref{conv_2_adically_corollary} in the sense of determining what other conditions the subsequences on which $r_{4,+}^{\ksymbol{k}}$ is bounded must satisfy. Determining sufficient conditions for boundedness would also be very interesting, though these may be more elusive.

\section*{Acknowledgments}
    This project was conducted as the author’s undergraduate thesis. We thank Krishnaswami Alladi for advising this thesis work and his vital mentorship. We also recognize Roger Heath-Brown and Trevor Wooley for their valuable advice, suggestions, and referrals to relevant literature.

\bibliography{main}{}
\bibliographystyle{amsplain}

\end{document}